\documentclass[12pt]{article}

\usepackage{amsmath}
\usepackage{amsthm}
\usepackage{amssymb}
\usepackage{cite}
\usepackage{url}
\usepackage[multiple]{footmisc}
\usepackage{graphicx}
\usepackage{epstopdf}
\usepackage{chngcntr}
\usepackage{enumitem}
\usepackage{hhline}
\usepackage{array}
\usepackage{hyperref}
\usepackage{enumitem}
\usepackage{color}

\hypersetup{colorlinks=false}

\if@twoside
\else 
\fi 

\numberwithin{equation}{section}
\counterwithin{figure}{section}
\counterwithin{table}{section}

\theoremstyle{plain}
\newtheorem{theorem}{Theorem}[section]
\newtheorem{lemma}[theorem]{Lemma}
\newtheorem{proposition}[theorem]{Proposition}
\newtheorem{corollary}[theorem]{Corollary}

\theoremstyle{definition}
\newtheorem{definition}{Definition}[section]

\newtheorem{example}{Example}[section]

\newtheorem{assumption}{Assumption}[section]

\theoremstyle{remark}
\newtheorem*{remark}{Remark}

\newcommand{\norm}[1]{\left\|#1\right\|}
\newcommand{\abs}[1]{\left\vert#1\right\vert}
\newcommand{\spr}[1]{\left\langle\,#1\,\right\rangle}
\newcommand{\kl}[1]{\left(#1\right)}

\newcommand{\LandauO}{\mathcal{O}}
\newcommand{\Landauo}{\textnormal{\scriptsize{$\mathcal{O}$}}}

\newcommand{\R}{\mathbb{R}} 

\newcommand{\N}{\mathbb{N}}
\newcommand{\Z}{\mathbb{Z}}

\newcommand{\X}{\mathcal{X}}
\newcommand{\Y}{\mathcal{Y}}

\newcommand{\xD}{x^\dagger}
\newcommand{\xs}{x_*}
\newcommand{\xd}{x^\delta}

\newcommand{\xad}{x_\alpha^\delta}
\newcommand{\xkd}{x_k^\delta}
\newcommand{\xkmd}{x_{k-1}^\delta}

\newcommand{\xkpd}{x_{k+1}^\delta}
\newcommand{\xmod}{x_{-1}^\delta}
\newcommand{\xksd}{x_{k_*}^\delta}
\newcommand{\xksdt}{\tilde{x}_{k_*}^{\delta}}

\newcommand{\xkndnt}{x_{\tilde{k}_n}^{\tilde{\delta}_n}}

\newcommand{\xt}{\tilde{x}}
\newcommand{\xzd}{x^\delta_0}
\newcommand{\xkz}{x_k^0}
\newcommand{\xkpz}{x_{k+1}^0}

\newcommand{\zkd}{z_{k}^\delta}
\newcommand{\zkz}{z_{k}^0}

\newcommand{\skd}{s_k^\delta}

\newcommand{\yd}{y^{\delta}}
\newcommand{\ydn}{y^{\delta_n}}

\newcommand{\D}{\mathcal{D}}

\newcommand{\Range}{\mathcal{R}}

\newcommand{\eps}{\varepsilon}
\newcommand{\Br}{\mathcal{B}_\rho}
\newcommand{\Btr}{\mathcal{B}_{2\rho}}
\newcommand{\mr}{6\rho}
\newcommand{\Bmr}{\mathcal{B}_{\mr}}
\newcommand{\ob}{\bar{\omega}}
\newcommand{\foh}{\tfrac{1}{2}}

\newcommand{\lka}{\tfrac{k-1}{k + \alpha - 1}}
\newcommand{\Phid}{\Phi^\delta}
\newcommand{\Phiz}{\Phi^0}

\newcommand{\Thetad}{\Theta^\delta}
\newcommand{\Thetaz}{\Theta^0}
\newcommand{\prox}[1]{\text{prox}_{\omega \Psi}\kl{#1}}

\newcommand{\Ed}{\mathcal{E}^\delta}
\newcommand{\Ez}{\mathcal{E}^0}
\newcommand{\wkd}{w^\delta_k}
\newcommand{\wkz}{w^0_k}
\newcommand{\wkpd}{w^\delta_{k+1}}
\newcommand{\Go}{G_{\omega}}
\newcommand{\God}{G_{\omega}^\delta}
\newcommand{\Goz}{G_{\omega}^0}

\newcommand{\Dd}{\Delta(\delta)}
\newcommand{\ks}{{k_*}}

\newcommand{\knt}{{\tilde{k}_n}}

\newcommand{\ak}{\alpha_k}
\newcommand{\akd}{\alpha_k^\delta}
\newcommand{\lkd}{\lambda^\delta_k}

\newcommand{\Tad}{\mathcal{T}_\alpha^\delta}

\renewcommand{\S}{\mathcal{S}}

\newcommand{\dn}{{\delta_n}}
\newcommand{\dnt}{{\tilde{\delta}_n}}

\newcommand{\Mt}{\mathcal{M}_\tau}

\newcommand{\bd}{\bar{\delta}}

\newcommand{\lt}{{\ell^2}}

\newcommand{\xv}{\vec{x}}

\newcommand{\LtO}{L^2(0,1)}
\newcommand{\HoO}{H^1(0,1)}
\newcommand{\Lt}{{L^2}}
\newcommand{\ek}{e^{(k)}}
\newcommand{\LN}{{\Lambda_N}}

\title{Nesterov's Accelerated Gradient Method for Nonlinear Ill-Posed Problems with a Locally Convex Residual Functional}
\author{
Simon Hubmer\footnote{Johannes Kepler University Linz, Doctoral Program Computational Mathematics, Altenbergerstra{\ss}e 69, A-4040 Linz, Austria (simon.hubmer@dk-compmath.jku.at)},
Ronny Ramlau\footnote{Johannes Kepler University Linz, Institute of Industrial Mathematics, Altenbergerstra{\ss}e 69, A-4040 Linz, Austria (ronny.ramlau@jku.at)} \footnote{Johann Radon Institute Linz, Altenbergerstra{\ss}e 69, A-4040 Linz, Austria (ronny.ramlau@ricam.oeaw.ac.at)}}

\begin{document}

\maketitle

\begin{abstract}
In this paper, we consider Nesterov's Accelerated Gradient method for solving Nonlinear Inverse and Ill-Posed Problems. Known to be a fast gradient-based iterative method for solving well-posed convex optimization problems, this method also leads to promising results for ill-posed problems. Here, we provide a convergence analysis for ill-posed problems of this method based on the assumption of a locally convex residual functional. Furthermore, we demonstrate the usefulness of the method on a number of numerical examples based on a nonlinear diagonal operator and on an inverse problem in auto-convolution.


\medskip
\noindent \textbf{AMS:} 65J15, 65J20, 65J22
\end{abstract}

\section{Introduction}
In this paper, consider nonlinear inverse problems of the form
\begin{equation}\label{Fx=y}
F(x) = y \,,
\end{equation}
where $F: \D(F) \subset \X \to \Y$ is a continuously Fr\'echet-differentiable, nonlinear operator between real Hilbert spaces $\X$ and $\Y$. Throughout this paper we assume that \eqref{Fx=y} has a solution $x_*$, which need not be unique. Furthermore, we assume that instead of $y$, we are only given noisy data $\yd$ satisfying
\begin{equation}\label{noisy_data}
\norm{y - \yd} \leq \delta \,.
\end{equation}
Since we are interested in ill-posed problems, we need to use regularization methods in order to obtain stable approximations of solutions of \eqref{Fx=y}. The two most prominent examples of such methods are \emph{Tikhonov regularization} and \emph{Landweber iteration}. 

In \emph{Tikhonov regularization}, one attempts to approximate an $x_0$-minimum-norm solution $\xD$ of \eqref{Fx=y}, i.e., a solution of $F(x)=y$ with minimal distance to a given initial guess $x_0$, by minimizing the functional
\begin{equation}\label{Tikhonov}
\Tad(x) := \norm{F(x)-\yd}^2 + \alpha \norm{x-x_0}^2 \,,
\end{equation}
where $\alpha$ is a suitably chosen regularization parameter. Under very mild assumptions on $F$, it can be shown that the minimizers of $\Tad$, usually denoted by $\xad$, converge subsequentially to a minimum norm solution $\xD$ as $\delta \to 0$, given that $\alpha$ and the noise level $\delta$ are coupled in an appropriate way \cite{Engl_Hanke_Neubauer_1996}. While for linear operators $F$ the minimization of $\Tad$ is straightforward, in the case of nonlinear operators $F$ the computation of $\xad$ requires the global minimization of the then also nonlinear functional $\Tad$, which is rather difficult and usually done using various iterative optimization algorithms.

This motivates the direct application of iterative algorithms for solving \eqref{Fx=y}, the most popular of which being \emph{Landweber iteration}, given by
\begin{equation}\label{Landweber}
\begin{split}
\xkpd &= \xkd + \omega F'(\xkd)^*(\yd - F(\xkd)) \,,
\\
\xzd &= x_0 \,,
\end{split}
\end{equation}
where $\omega$ is a scaling parameter and $x_0$ is again a given initial guess. Seen in the context of classical optimization algorithms, Landweber iteration is nothing else than the gradient descent method applied to the functional 
\begin{equation}\label{def_phi}
\Phid(x) := \foh \norm{F(x) - \yd}^2 \,,
\end{equation}
and therefore, in order to arrive at a convergent regularization method, one has to use a suitable stopping rule. In \cite{Engl_Hanke_Neubauer_1996} it was shown that if one uses the discrepancy principle, i.e., stops the iteration after $k_*$ steps, where $k_*$ is the smallest integer such that
\begin{equation}\label{discrepancy_principle}
\norm{\yd - F(\xksd)} \leq \tau \delta < \norm{\yd - F(\xkd)} \,, \qquad 0 \leq k < k_* \,,
\end{equation}
with a suitable constant $\tau > 1$, then Landweber iteration gives rise to a convergent regularization method, as long as some additional assumptions, most notably the (strong) tangential cone condition 
\begin{equation}\label{cond_nonlin}
\begin{split}
&\norm{F(x)-F(\xt)-F'(x)(x-\xt)} \leq \eta \norm{F(x) - F(\xt)} \,, \qquad \eta < \foh
\\
& x,\xt \in \Btr(x_0) \,,
\end{split} 
\end{equation}
where $\Btr(x_0)$ denotes the closed ball of radius $2 \rho $ around $x_0$, is satisfied. Since condition \eqref{cond_nonlin} poses strong restrictions on the nonlinearity of $F$ which are not always satisfied, attempts have been made to use weaker conditions instead \cite{Scherzer_1995}. For example, assuming only the weak tangential cone condition
\begin{equation}\label{cond_nonlin_weak}
\begin{split}
\spr{F(x)-F(\xs) - F'(x)(x-\xs),F(x)-F(\xs)} \leq \eta \norm{F(x)-F(\xs)}^2 
\\
\forall x \in \Br(x_0) \,, \qquad 0 < \eta < 1 \,, 
\end{split}
\end{equation}
to hold, one can show weak convergence of Landweber iteration \cite{Scherzer_1995}. Similarly, if the residual functional $\Phiz(x)$ defined by \eqref{def_phi} is (locally) convex, weak subsequential convergence of the iterates of Landweber iteration to a stationary point of $\Phiz(x)$ can be proven. Even though they both lead to convergence in the weak topology, besides some results presented in \cite{Scherzer_1995}, the connections between the local convexity of the residual functional and the (weak) tangential cone condition remain largely unexplored. In his recent paper \cite{Kindermann_2017}, Kindermann showed that both the local convexity of the residual functional and the weak tangential cone condition imply another condition, which he termed $NC(0,\beta>0)$, and which is sufficient to guarantee weak subsequential convergence of the iterates.

As is well known, Landweber iteration is quite slow \cite{Kaltenbacher_Neubauer_Scherzer_2008}. Hence, acceleration strategies have to be used in order to speed it up and make it applicable in practise. Acceleration methods and their analysis for linear problems can be found for example in \cite{Engl_Hanke_Neubauer_1996} and \cite{Hanke_1991}. Unfortunately, since their convergence proofs are mainly based on spectral theory, their analysis cannot be generalized to nonlinear problems immediately. However, there are some acceleration strategies for Landweber iteration for nonlinear ill-posed problems, for example \cite{Neubauer_2000, Ramlau_1999}.

As an alternative to (accelerated) Landweber-type methods, one could think of using second order iterative methods for solving \eqref{Fx=y}, such as the Levenberg-Marquardt method \cite{Hanke_1997,Jin_2010}
\begin{equation}\label{LM}
\xkpd = \xkd + (F'(\xkd)^*F'(\xkd) + \ak I)^{-1} F'(\xkd)^*(\yd - F(\xkd)) \,,
\end{equation}
or the iteratively regularized Gauss-Newton method \cite{Blaschke_Neubauer_Scherzer_1997,Jin_Tautenhahn_2009}
\begin{equation}\label{irGN}
\xkpd = \xkd + (F'(\xkd)^*F'(\xkd) + \ak I)^{-1}( F'(\xkd)^*(\yd - F(\xkd)) + \ak(x_0 - \xkd)) \,.
\end{equation} 
The advantage of those methods \cite{Kaltenbacher_Neubauer_Scherzer_2008} is that they require much less iterations to meet their respective stopping criteria compared to Landweber iteration or the steepest descent method. However, each update step of those iterations might take considerably longer than one step of Landweber iteration, due to the fact that in both cases a linear system involving the operator
\begin{equation*}
F'(\xkd)^*F'(\xkd) + \ak I
\end{equation*} 
has to be solved. In practical applications, this usually means that a huge linear system of equations has to be solved, which often proves to be costly, if not infeasible. Hence, accelerated Landweber type methods avoiding this drawback are desirable in practise.

In case that the residual functional $\Phid(x)$ is locally convex, one could think of using methods from convex optimization to minimize $\Phid(x)$, instead of using the gradient method like in Landweber iteration. One of those methods, which works remarkably well for nonlinear, convex and well-posed optimization problems of the form
\begin{equation}\label{prob_conv_min}
\min\{\Phi(x) \, \vert \, x \in \X \} 
\end{equation}
was first introduced by Nesterov in \cite{Nesterov_1983} and is given by
\begin{equation}\label{N_1}
\begin{split}
z_k &= x_k + \tfrac{k-1}{k + \alpha - 1} (x_k - x_{k-1}) \,,
\\
x_{k+1} &= z_k - \omega(\nabla \Phi(z_k))  \,,
\end{split}
\end{equation}
where again $\omega$ is a given scaling parameter and $\alpha \geq 3$ (with $\alpha = 3$ being common practise). This so-called \emph{Nesterov acceleration scheme} is of particular interest, since not only is it extremely easy to implement, but Nesterov himself was also able to prove that it generates a sequence of iterates $x_k$ for which there holds 
\begin{equation}\label{res_speed_1}
\norm{\Phi(x_k) - \Phi(x_*)} = \LandauO(k^{-2}) \,,
\end{equation}
where $x_*$ is any solution of \eqref{prob_conv_min}. This is a big improvement over the classical rate $\LandauO(k^{-1})$. The even further improved rate $\Landauo(k^{-2})$ for $\alpha > 3$ was recently proven in \cite{Attouch_Peypouquet_2016}.

Furthermore, Nesterov's acceleration scheme can also be used to solve compound optimization problems of the form
\begin{equation}\label{prob_min_comp}
\min\{\Phi(x) + \Psi(x) \, \vert \, x \in \X \} \,,
\end{equation}
where both $\Phi(x)$ and $\Psi(x)$ are convex functionals, and is in this case given by
\begin{equation}\label{N_2}
\begin{split}
z_k &= x_k + \tfrac{k-1}{k + \alpha - 1} (x_k - x_{k-1}) \,,
\\
x_{k+1} &= \prox{z_k - \omega(\nabla \Phi(z_k)) }  \,,
\end{split}
\end{equation}
where the proximal operator $\prox{.}$ is defined by
\begin{equation}\label{def_prox}
\prox{x} := \arg\min_u \left\{\omega \Psi(u) + \foh \norm{x - u}^2 \right\}\,.
\end{equation}
If in addition to being convex, $\Psi$ is proper and lower-semicontinous and $\Phi$ is continuously Fr\'echet differentiable with a Lipschitz continuous gradient, then it was again shown in \cite{Attouch_Peypouquet_2016} that the sequence defined by \eqref{N_2} satisfies
\begin{equation}\label{res_speed_2}
\norm{(\Phi - \Psi)(x_k) - (\Phi-\Psi)(\xs) } = \LandauO(k^{-2}) \,,
\end{equation}
or even $\Landauo(k^{-2})$ if $\alpha > 3$, which is again much faster than ordinary first order methods for minimizing \eqref{prob_min_comp}. This accelerating property was exploited in the highly successful FISTA algorithm \cite{Beck_Teboulle_2009}, designed for the fast solution of linear ill-posed problems with sparsity constraints. Since for linear operators the residual functional $\Phid$ is globally convex, minimizing the resulting Tikhonov functional \eqref{Tikhonov} exactly fits into the category of minimization problems considered in \eqref{N_2}.   

Motivated by the above considerations, one could think of applying Nesterov's acceleration scheme \eqref{N_1} to the residual functional $\Phid$, which leads to the algorithm
\begin{equation}\label{Nesterov_1}
\begin{split}
\zkd & = \xkd + \tfrac{k - 1}{k + \alpha - 1}(\xkd  - \xkmd) \,,
\\
\xkpd & = \zkd + \omega F'(\zkd)^*(\yd-F(\zkd)) \,,
\\
\xzd &= \xmod = x_0 \,.
\end{split}
\end{equation}
In case that the operator $F$ is linear, Neubauer showed in \cite{Neubauer_2017} that, combined with a suitable stopping rule and under a source condition, \eqref{Nesterov_1} gives rise to a convergent regularization method and that convergence rates can be obtained. Furthermore, the authors of \cite{Hubmer_Ramlau_2017} showed that certain generalizations of Nesterov's acceleration scheme, termed \emph{Two-Point Gradient (TPG) methods} and given by
\begin{equation}\label{Two_Point}
\begin{split}
\zkd &= \xkd + \lkd(\xkd - \xkmd) \,,
\\
\xkpd &= \zkd + \akd\skd \,, \qquad \skd := F'(\zkd)^*(\yd - F(\zkd)) \,,
\\
\xzd &= \xmod = x_0 \,,
\end{split}
\end{equation}
give rise to convergent regularization methods, as long as the tangential cone condition ~\eqref{cond_nonlin} is satisfied and the stepsizes $\akd$ and the combination parameters $\lkd$ are coupled in a suitable way. However, the convergence analysis of the methods \eqref{Two_Point} does not cover the choice
\begin{equation}\label{lambda_N}
\lkd = \frac{k-1}{k+\alpha-1} \,,
\end{equation} 
i.e., the choice originally proposed by Nesterov and the one which shows by far the best results numerically \cite{Hubmer_Neubauer_Ramlau_Voss_2018, Hubmer_Ramlau_2017,Jin_2016}. The main reason for this is that the techniques employed there works with the monotonicity of the iteration, i.e., the iterate $\xkpd$ always has to be a better approximation of the solution $\xs$ than $\xkd$, which is not necessarily satisfied for the choice \eqref{lambda_N}. 

The key ingredient for proving the fast rates \eqref{res_speed_1} and \eqref{res_speed_2} is the convexity of the residual functional $\Phi(x)$. Since, except for linear operators, we cannot hope that this holds globally, we assume that $\Phiz(x)$, i.e., the functional $\Phid(x)$ defined by \eqref{def_phi} with exact data $y = \yd$, corresponding to $\delta = 0$, is convex in a neighbourhood of the initial guess. This neighbourhood has to be sufficiently large encompassing the sought solution $\xs$, or equivalently, the initial guess $x_0$ has to be sufficiently close to the solution $\xs$. Assuming that $F(x) =y$ has a solution $\xs$ in $\Br(x_0)$, where now and in the following, $\Br(x_0)$ denotes the closed ball with radius $\rho$ around $x_0$, the key assumption is that $\Phiz$ is convex in $\Bmr(x_0)$. As mentioned before, Nesterov's acceleration scheme yields a non-monotonous sequence of iterates, which might possible leave the ball $\Bmr(x_0)$. However, by assumption the sought for solution $\xs$ lies in the ball $\Br(x_0)$. Hence, defining the functional 
\begin{equation}\label{def_psi}
\Psi(x) :=
\begin{cases}
0 \,, & x \in \Btr(x_0) \,,
\\
\infty \,, & x \notin \Btr(x_0) \,,
\end{cases}
\end{equation}
we can, instead of using \eqref{N_1}, which would lead to algorithm \eqref{Nesterov_1}, use \eqref{N_2}, noting that still the fast rate \eqref{res_speed_2} can be expected for $\delta=0$. This leads to the algorithm
\begin{equation}\label{Nesterov}
\begin{split}
\zkd &= \xkd + \lka(\xkd - \xkmd) \,,
\\
\xkpd &= \prox{\zkd + \omega F'(\zkd)^*(\yd-F(\zkd))} \,,
\\
\xd_0 &= \xd_{-1} = x_0 \,,
\end{split}
\end{equation}
which we consider throughout this paper.

\section{Convergence Analysis I}

In this section we provide a convergence analysis of Nesterov's accelerated gradient method \eqref{Nesterov}. Concerning notation, whenever we consider the noise-free case $y = \yd$ corresponding to $\delta = 0$, we replace $\delta$ by $0$ in all variables depending on $\delta$, e.g., we write $\Phiz$ instead of $\Phid$. For carrying out the analysis, we have to make a set of assumptions, already indicated in the introduction.
\begin{assumption}\label{ass_main}
Let $\rho$ be a positive number such that $\Bmr(x_0) \subset \D(F)$.
\begin{enumerate}
\item The operator $F : \D(F) \subset \X \to \Y$ is continuously Fr\'echet differentiable between the real Hilbert spaces $\X$ and $\Y$ with inner products $\spr{.,.}$ and norms $\norm{.}$. Furthermore, let $F$ be weakly sequentially closed on $\Btr(x_0)$.
\item The equation $F(x) = y$ has a solution $\xs \in \Br(x_0)$.
\item The data $\yd$ satisfies $\norm{y-\yd} \leq \delta$.
\item The functional $\Phi^0$ defined by \eqref{def_phi} with $\delta = 0$ is convex and has a Lipschitz continuous gradient $\nabla \Phi^0$ with Lipschitz constant $L$ on $\Bmr(x_0)$, i.e.,
\begin{equation}\label{cond_convex}
\Phi^0(\lambda x_1 + (1-\lambda) x_2) \leq \lambda \Phi^0(x_1) + (1-\lambda) \Phi^0(x_2)  
\,, 
\qquad \forall x_1, x_2 \in \Bmr(x_0) \,,
\end{equation} 
\begin{equation}\label{cond_L_grad}
\norm{\nabla \Phi^0(x_1) - \nabla \Phi^0(x_2)} \leq L \norm{x_1 - x_2}
\,,
\qquad \forall x_1, x_2 \in \Bmr(x_0) \,.
\end{equation}
\item For $\alpha$ in \eqref{Nesterov} there holds $\alpha > 3$ and the scaling parameter $\omega$ satisfies $0 < \omega < \frac{1}{L}$.
\end{enumerate}
\end{assumption}

Note that since $\Btr(x_0)$ is weakly closed and given the continuity of $F$, a sufficient condition for the weak sequential closedness assumption to hold is that $F$ is compact.

We now turn to the convergence analysis of Nesterov's accelerated gradient method \eqref{Nesterov}. Throughout this analysis, if not explicitly stated otherwise, Assumption~\ref{ass_main} is in force. Note first that from $F$ being continuously Fr\'echet differentiable, we can derive that there exists an $\ob$ such that
\begin{equation}\label{cond_bounded}
\norm{F'(x)} \leq \ob \,, \qquad \forall x \in \Bmr(x_0) \,.
\end{equation}

Next, note that since $\Btr(x)$ denotes a closed ball around $x$, the functional $\Psi$, in addition to being proper and convex, is also lower-semicontinous, an assumption required in the proofs in \cite{Attouch_Peypouquet_2016}, which we need in various places of this paper. Furthermore, it immediately follows from the definition \eqref{def_prox} of the proximal operator $\prox{.}$ that 
\begin{equation}\label{prox_proj}
\begin{split}
\prox{x} = \arg\min_{u \in \X} \left\{\omega \Psi(u) + \foh \norm{x - u}^2 \right\}
=\arg \hspace{-10pt}\min_{u \in \Btr(x_0)} \left\{\foh \norm{x - u}^2 \right\} \,,
\end{split}
\end{equation}  
since $\Psi$ defined by \eqref{def_psi} is equal to $\infty$ outside $\Btr(x_0)$. Hence, since obviously $\Btr(x_0)$ is a convex set, $\prox{.}$ is nothing else than the metric projection onto $\Btr(x_0)$, and is therefore Lipschitz continuous with Lipschitz constant smaller or equal to $1$. Consequently, given an estimate of $\rho$, the implementation of $\prox{.}$ is exceedingly simple in this setting, and therefore, one iteration step of \eqref{Nesterov} and \eqref{Landweber} require roughly the same amount of computational effort.

Finally, note that due to the convexity of $\Phiz$, the set $S$ defined by
\begin{equation}\label{def_S}
\S := \left\{ x \in \Btr(x_0) \, \vert \, F(x) = y   \right\} \,,
\end{equation}
is a convex subset of $\Btr(x_0)$ and hence, there exists a unique $x_0$-minimum-norm solution $\xD$, which is defined by
\begin{equation}\label{def_xD}
\xD := \arg \min_{x \in \S }  \norm{x - x_0} \,,
\end{equation}
which is nothing else than the orthogonal projection of $x_0$ onto the set $\S$.

The following convergence analysis is largely based on the ideas of the paper \cite{Attouch_Peypouquet_2016} of Attouch and Peypouquet, which we reference from frequently throughout this analysis. Following their arguments, we start by making the following
\begin{definition}\label{Def_stuff_1}
For $\Phid$ and $\Psi$ defined by \eqref{def_phi} and $\eqref{def_psi}$, we define
\begin{equation}\label{def_theta}
\Thetad (x) := \Phid(x) + \Psi(x) \,.
\end{equation}
The energy functional $\Ed $ is defined by
\begin{equation}\label{def_E}
\Ed(k) := \frac{2 \omega}{\alpha-1} (k+\alpha-2)^2 (\Thetad(\xkd)-\Thetad(\xs)) + (\alpha - 1)\norm{\wkd-\xs}^2 \,,
\end{equation}
where the sequence $\wkd$ is defined by
\begin{equation}\label{def_w}
\wkd := \frac{k+\alpha-1}{\alpha-1}\zkd - \frac{k}{\alpha-1} \xkd
= \xkd + \frac{k-1}{\alpha-1}\left(\xkd - \xkmd \right) \,.
\end{equation}
Furthermore, we introduce the operator $\God : \D(F) \subset \X \to \Y$, given by
\begin{equation}\label{def_G}
\God(x) := \frac{1}{\omega}\left(x - \prox{x - \omega \nabla \Phid(x)}\right) \,.
\end{equation}
\end{definition}

Using Definition~\ref{Def_stuff_1}, we can now write to update step for $\xkpd$ in the form
\begin{equation*}
\xkpd = \zkd - \omega \, \God(\zkd) \,,
\end{equation*}
and furthermore, it is possible to write
\begin{equation}\label{eq_wk}
\wkpd = \frac{k+\alpha-1}{\alpha-1}\left( \zkd - \omega \, \God(\zkd) \right) - \frac{k}{\alpha-1} \xkd
= 
\wkd - \frac{\omega}{\alpha-1} (k+\alpha-1) \God(\zkd)\,.
\end{equation}
As a first result, we show that both $\zkd$ and $\xkd$ stay within $\Bmr(x_0)$ during the iteration.
\begin{lemma}\label{lem_within}
Under the Assumption~\ref{ass_main}, the sequence of iterates $\xkd$ and $\zkd$ defined by \eqref{Nesterov} is well-defined. Furthermore, $\xkd \in \Btr(x_0)$ and $\zkd \in \Bmr(x_0)$ for all $k \in \N$.
\end{lemma}
\begin{proof}
This follows by induction from $\xd_0 = \xd_{-1} = x_0 \in \Br(x_0)$, the observation
\begin{equation*}
\begin{split}
\norm{\zkd - x_0} &\leq (1+\lka)\norm{\xkd - x_0} + \lka \norm{\xkmd - x_0} 
\\
&\leq 2\norm{\xkd - x_0} + \norm{\xkmd - x_0} \,,
\end{split}
\end{equation*}
and the fact that by the definition of $\prox{x}$, $\xkd$ is always an element of $\Btr(x_0)$.
\end{proof}

Since the functional $\Theta^0$ is assumed to be convex in $\Bmr(x_0)$, we can deduce:
\begin{lemma}\label{lem_main_zero}
Under Assumption~\ref{ass_main}, for all $x,z \in \Bmr(x_0)$ there holds
\begin{equation*}
\Theta^0(z-\omega \Go^0(z)) \leq \Theta^0(x)+ \spr{\Go^0(z),z-x} - \frac{\omega}{2} \norm{\Go^0(z)}^2 \,.
\end{equation*}
\end{lemma}
\begin{proof}
This lemma is also used in \cite{Attouch_Peypouquet_2016}. However, the sources for it cited there do not exactly cover our setting with $\Phid$ being defined on $\D(F) \subset \X$ only. Hence, we here give an elementary proof of the assertion. Note first that due to the Lipschitz continuity of $\Phiz$ in $\Bmr(x_0)$ and the fact that $\omega < 1/L$ we have
\begin{equation*}
\Phiz(u) \leq \Phiz(v) + \spr{\nabla \Phiz(v),u-v} + \frac{1}{2\omega} \norm{u-v}^2 \,,
\qquad
\forall \, u , v \in \Bmr(x_0) \,.
\end{equation*}
Now since $\Phiz$ is convex on $\Bmr(x_0)$, also have \cite{Bauschke_Combettes_2017}
\begin{equation*}
\Phiz(v) + \spr{\nabla \Phiz(v),w-v} \leq \Phiz(w) \,,
\qquad
\forall \, v , w \in \Bmr(x_0) \,,
\end{equation*}
and therefore, combining the above two inequalities, we get
\begin{equation*}
\Phiz(u) \leq \Phiz(w) + \spr{\nabla \Phiz(v),u-w} + \frac{1}{2\omega} \norm{u-v}^2 \,,
\qquad
\forall \, u ,v, w \in \Bmr(x_0) \,.
\end{equation*}
Using this result for $u = z-\omega \Go^0(z)$, $v = z$, $w = x$, noting that for $x,z \in \Bmr(x_0)$ there holds $u,v,w \in \Bmr(x_0)$, we get
\begin{equation}\label{ineq_Phiz_descent}
\Phiz(z-\omega \Go^0(z)) \leq \Phiz(x) + \spr{\nabla \Phiz(z),z-\omega \Go^0(z)-x} + \frac{\omega}{2} \norm{\Go^0(z)}^2 \,. 
\end{equation}
Next, note that since $z-\omega \Go^0(z) = \prox{z- \omega \nabla \Phiz(z)}$, a standard result from proximal operator theory \cite[Proposition~12.26]{Bauschke_Combettes_2017} implies that there holds
\begin{equation*}
\begin{split}
&\Psi(z-\omega \Go^0(z)) 
\leq
\Psi(x) 
+
\frac{1}{\omega}\spr{(z-\omega \Go^0(z) )-x,
(z- \omega \nabla \Phiz(z)) -(z-\omega \Go^0(z) )} 
\\
& \qquad =
\Psi(x) 
+
\spr{z-\omega \Go^0(z) -x,
- \nabla \Phiz(z)+ \Go^0(z) } 
\\
& \qquad =
\Psi(x) 
-
\spr{z-\omega \Go^0(z) -x,
\nabla \Phiz(z) }
+ 
\spr{z-x,\Goz(z)}
- \omega \norm{\Goz(z)}^2 
\,.
\end{split}
\end{equation*}
Adding this inequality to \eqref{ineq_Phiz_descent} and using the fact that by definition $\Thetaz = \Phiz + \Psi$ immediately yields the assertion.
\end{proof}

We want to derive a similar inequality also for the functionals $\Thetad$. The following lemma is of vital importance for doing that:
\begin{lemma}\label{lem_main_noise}
Let Assumption~\ref{ass_main} hold, let $x,z \in \Bmr(x_0)$ and define
\begin{equation}\label{def_Ri}
\begin{split}
R_1 &:= \Theta^0(z - \omega \God(z)) - \Theta^0(z-\omega \Go^0(z)) \,,
\\
R_2 &:=  \Thetad(z - \omega \God(z)) - \Theta^0(z-\omega \God(z))\,,
\\
R_3 &:= \Theta^0(x) - \Thetad(x)  \,,
\quad
R_4 := \spr{\Go^0(z) - \God(z),z-x} \,,
\\
R_5 &:= \frac{\omega}{2} \left(\norm{\God(z)}^2 - \norm{\Go^0(z)}^2 \right) \,,
\end{split}
\end{equation}
as well as 
\begin{equation}\label{def_R}
R := R_1 + R_2 + R_3 + R_4 + R_5 \,.
\end{equation}
Then there holds
\begin{equation*}
\Thetad(z-\omega \God(z)) \leq \Thetad(x)+ \spr{\God(z),z-x} - \frac{\omega}{2} \norm{\God(z)}^2 + R \,.
\end{equation*}
\end{lemma}
\begin{proof}
Using Lemma~\ref{lem_main_zero} we get
\begin{equation*}
\begin{split}
&\Thetad(z-\omega \God(z)) 
= 
\Thetaz(z-\omega \Goz(z)) + R_1 + R_2 
\\
& \qquad \leq 
\Theta^0(x)+ \spr{\Go^0(z),z-x} - \frac{\omega}{2} \norm{\Go^0(z)}^2 + R_1 + R_2
\\
& \qquad = 
\Thetad(x)+ \spr{\God(z),z-x} - \frac{\omega}{2} \norm{\God(z)}^2 + R_1 + R_2 + R_3 + R_4 + R_5 \,,
\end{split}
\end{equation*}
from which the statement of the theorem immediately follows.
\end{proof}

Next, we show that the $R_i$ and hence, also $R$, can be bounded in terms of $\delta + \delta^2$.
\begin{proposition}\label{prop_R_delta}
Let Assumption~\ref{ass_main} hold, let $x \in \Btr(x_0)$ and $z \in \Bmr(x_0)$ and let the $R_1,\dots,R_5$ be defined by \eqref{def_R}. Then there holds
\begin{equation*}
\begin{split}
R_1 & \leq \foh \ob^4 \omega^2  \delta^2 + 2 \ob^3 \omega \rho \delta  \,,
\\
R_2 &\leq \tfrac{3}{2} \delta^2 + 2 \rho \, \ob \, \delta \,,
\\
R_3 &\leq \tfrac{3}{2} \delta^2 + 2 \rho \, \ob \, \delta \,,
\\
R_4 &\leq 8 \rho \, \ob \delta\,,
\\
R_5 & \leq 
\frac{\omega}{2}\ob^2 \, \delta^2 + 8 \rho \, \ob \, \delta \,. 
\end{split}
\end{equation*}
\end{proposition}
\begin{proof}
The following somewhat long but elementary proof uses mainly the boundedness and Lipschitz continuity assumptions made above. For the following, let $x \in \Btr(x_0)$ and $z \in \Bmr(x_0)$. We treat each of the $R_i$ terms separately, starting with 
\begin{equation*}
\begin{split}
R_1 &= \Theta^0(z - \omega \God(z)) - \Theta^0(z-\omega \Go^0(z))
\\
& =
\foh \norm{F(z - \omega \God(z)) - y}^2 - \foh \norm{F(z - \omega \Goz(z))-y}^2
\\
& =
\foh \norm{F(z - \omega \God(z)) - F(z - \omega \Goz(z))}^2 \\
& \qquad - \spr{F(z - \omega \God(z)) - F(z - \omega \Goz(z)),F(z - \omega \Goz(z))-y}
\\
& \leq
\foh \norm{F(z - \omega \God(z)) - F(z - \omega \Goz(z))}^2 \\
& \qquad + \norm{F(z - \omega \God(z)) - F(z - \omega \Goz(z))}\norm{F(z - \omega \Goz(z))-y} \,.
\end{split}
\end{equation*}
Since we have
\begin{equation*}
\begin{split}
&\norm{F(z - \omega \God(z)) - F(z - \omega \Goz(z))} \leq \ob \norm{\omega \God(z) - \omega \Goz(z)}
\\
& \qquad
\leq \ob \norm{\prox{z-\omega\nabla \Phid(z) } - \prox{z-\omega\nabla \Phi(z) }}
\\
& \qquad
\leq \ob \, \omega \norm{\nabla \Phid(z)  - \nabla \Phi(z) } 
\\
& \qquad
= \ob \, \omega \norm{F'(z)^*(y-\yd) } \leq \ob^2 \, \omega \norm{y-\yd} \leq \ob^2 \, \omega \, \delta \,,
\end{split}
\end{equation*}
and
\begin{equation*}
\norm{F(z - \omega \Goz(z))-y} \leq \ob  \norm{\prox{z-\omega\nabla \Phid(z) } - \xs } \leq 2 \rho \, \ob \,,
\end{equation*}
there holds
\begin{equation*}
R_1 \leq \foh (\ob^2 \, \omega \, \delta)^2 + (\ob^2 \, \omega \, \delta) 2 \rho \, \ob  = \left(\foh \ob^4 \omega^2 \right) \delta^2 + \left(2 \ob^3 \omega \rho\right) \delta \,.
\end{equation*}
Next, we look at
\begin{equation*}
\begin{split}
R_2 &=  \Thetad(z - \omega \God(z)) - \Theta^0(z-\omega \God(z))
\\
& = \foh \norm{y - \yd}^2 + \spr{F(z - \omega \God(z))-\yd, y- \yd}
\\
& = \tfrac{3}{2} \norm{y - \yd}^2 + \spr{F(z - \omega \God(z))-y, y- \yd}
\\
& \leq \tfrac{3}{2} \delta^2 + \norm{F(z - \omega \God(z))-F(\xs)} \delta
\\
& \leq \tfrac{3}{2} \delta^2 + 2 \rho \, \ob \, \delta \,.
\end{split}
\end{equation*}
Similarly to above, for the next term we get
\begin{equation*}
\begin{split}
R_3 &= \Theta^0(x) - \Thetad(x)  = \foh \norm{F(x)-y}^2 - \foh \norm{F(x)-\yd}^2
\\
& = \foh \norm{y - \yd}^2 + \spr{F(x)-\yd, y- \yd}
\\
& = \tfrac{3}{2} \norm{y - \yd}^2 + \spr{F(x)-y, y- \yd}
\\
& \leq \tfrac{3}{2} \delta^2 + \norm{F(x)-F(\xs)} \delta
\\
& \leq \tfrac{3}{2} \delta^2 + 2 \rho \, \ob \, \delta \,.
\end{split}
\end{equation*}
Furthermore, together with the Lipschitz continuity of $\prox{.}$, we get
\begin{equation*}
\begin{split}
R_4 &= \spr{\Go^0(z) - \God(z),z-x} 
\\
& =
\frac{1}{\omega}\spr{\prox{z-\omega \nabla \Phid(z)}-\prox{z-\omega \nabla \Phi^0(z)},z-x}
\\
& \leq
\frac{1}{\omega}\norm{\prox{z-\omega \nabla \Phid(z)}-\prox{z-\omega \nabla \Phi^0(z)}} \norm{z-x}
\\
& \leq
\norm{ \nabla \Phid(z)-  \nabla \Phi^0(z)} \norm{z-x}
\leq
8 \rho \norm{ F'(z)(y-\yd)} \leq 8 \rho \, \ob \delta \,.
\end{split}
\end{equation*}
Finally, for the last term, we get
\begin{equation*}
\begin{split}
R_5 &= \frac{\omega}{2} \left(\norm{\God(z)}^2 - \norm{\Go^0(z)}^2 \right) 
\\
& = 
\frac{\omega}{2} \norm{\God(z)-\Goz(z)}^2 + \omega \spr{\God(z)-\Goz(z) , \Goz(z)}
\\
& \leq 
\frac{\omega}{2} \norm{\God(z)-\Goz(z)}^2 + \omega \norm{\God(z)-\Goz(z)} \norm{\Goz(z)}
\\
& \leq 
\frac{\omega}{2}\ob^2 \, \delta^2 + \omega \, \ob \, \delta \norm{\Goz(z)} 
\leq 
\frac{\omega}{2}\ob^2 \, \delta^2 + 8 \rho \, \ob \, \delta \,,
\end{split}
\end{equation*}
which concludes the proof.
\end{proof}

As an immediate consequence, we get the following
\begin{corollary}\label{cor_main_noise}
Let Assumption~\ref{ass_main} hold and let $x,z \in \Bmr(x_0)$. If we define
\begin{equation}\label{def_c1_c2}
\begin{split}
c_1 &= 2 \, \ob^3 \,\omega\, \rho +  20 \rho \, \ob \,,
\\
c_2 &= 3 + \foh \ob^4 \omega^2 + \foh \omega\ob^2 \,,
\end{split}
\end{equation}
then there holds
\begin{equation*}
\Thetad(z-\omega \God(z)) \leq \Thetad(x)+ \spr{\God(z),z-x} - \frac{\omega}{2} \norm{\God(z)}^2 + c_1 \delta + c_2 \delta^2\,.
\end{equation*}
\end{corollary}
\begin{proof}
This immediately follows from Lemma~\ref{lem_main_zero} and Proposition~\ref{prop_R_delta}.
\end{proof}

Combining the above, we are now able to arrive at the following important result:
\begin{proposition}\label{prop_descent}
Let Assumption~\ref{ass_main} hold, let the sequence of iterates $\xkd$ and $\zkd$ be given by \eqref{Nesterov} and let $c_1$ and $c_2$ be defined by \eqref{def_c1_c2}. If we define
\begin{equation}\label{def_Dd}
\Dd := c_1 \delta + c_2 \delta^2 \,,
\end{equation}
then there holds
\begin{align}
&\Thetad(\zkd - \omega \God(\zkd)) \leq \Thetad(\xkd) + \spr{\God(\zkd),\zkd-\xkd} - \frac{\omega}{2} \norm{\God(\zkd)}^2 + \Dd \label{ineq_main_1} \,,
\\
&\Thetad(\zkd - \omega \God(\zkd)) \leq \Thetad(\xs) + \spr{\God(\zkd),\zkd-\xs} - \frac{\omega}{2} \norm{\God(\zkd)}^2 + \Dd \label{ineq_main_2} \,.
\end{align}
\end{proposition}
\begin{proof}
This immediately follows from Lemma~\ref{lem_within} and Corollary~\ref{cor_main_noise}.
\end{proof}

Using the above proposition, we are now able to derive the important
\begin{theorem}
Let Assumption~\ref{ass_main} hold and let the sequence of iterates $\xkd$ and $\zkd$ be given by \eqref{Nesterov} and let $\Delta(\delta)$ be defined by \eqref{def_Dd}. Then there holds
\begin{equation}\label{ineq_main_Ed}
\Ed(k+1) + \frac{2 \omega}{\alpha-1} \left( k(\alpha-3) \left( \Thetad(\xkd) -  \Thetad(\xs) \right) - (k+\alpha-1)^2 \Dd \right)
\leq \Ed(k)  \,.
\end{equation}
\end{theorem}
\begin{proof}
This proof is adapted from the corresponding result in \cite{Attouch_Peypouquet_2016}, the difference being the term $\Dd$. We start by multiplying inequality \eqref{ineq_main_1} by $\frac{k}{k+\alpha-1}$ and inequality \eqref{ineq_main_2} by $\frac{\alpha-1}{k+\alpha-1}$. Adding the results and using the fact that $\xkpd = \zkd - \omega \God(\zkd)$, we get
\begin{equation*}
\begin{split}
\Thetad(\xkpd) \leq & \frac{k}{k+\alpha-1} \Thetad(\xkd) + \frac{\alpha-1}{k+\alpha-1} \Thetad(\xs) - \frac{\omega}{2} \norm{\God(\zkd)}^2 + \Dd
\\
& + \spr{\God(\zkd), \frac{k}{k+\alpha-1}  (\zkd - \xkd) + \frac{\alpha-1}{k+\alpha-1} (\zkd - \xs)} \,.
\end{split}
\end{equation*} 
Since
\begin{equation*}
\frac{k}{k+\alpha-1}  (\zkd - \xkd) + \frac{\alpha-1}{k+\alpha-1} (\zkd - \xs) = \frac{\alpha-1}{k+\alpha-1} (\wkd - \xs) \,,
\end{equation*}
we obtain
\begin{equation}\label{ineq_helper_1}
\begin{split}
\Thetad(\xkpd) \leq & \frac{k}{k+\alpha-1} \Thetad(\xkd) + \frac{\alpha-1}{k+\alpha-1} \Thetad(\xs) - \frac{\omega}{2} \norm{\God(\zkd)}^2 + \Dd
\\
& \frac{\alpha-1}{k+\alpha-1} \spr{\God(\zkd), \wkd - \xs} \,.
\end{split}
\end{equation} 
Next, observe that it follows from \eqref{eq_wk} that
\begin{equation*}
\wkpd - \xs = \wkd - \xs - \frac{\omega}{\alpha - 1}(k+\alpha -1) \God(\zkd) \,.
\end{equation*}
After developing
\begin{equation*}
\begin{split}
\norm{\wkpd - \xs}^2 &= \norm{\wkd - \xs}^2 - 2  \frac{\omega}{\alpha - 1}(k+\alpha -1) \spr{\wkd - \xs, \God(\zkd)} 
\\
&+ \frac{\omega^2}{(\alpha - 1)^2}(k+\alpha -1)^2 \norm{\God(\zkd)}^2 \,, 
\end{split}
\end{equation*}
and multiplying the above expression by $ \frac{(\alpha - 1)^2}{2\omega (k+\alpha -1)^2}$, we get
\begin{equation*}
\begin{split}
& \frac{(\alpha - 1)^2}{2\omega (k+\alpha -1)^2} 
 \left( \norm{\wkd - \xs}^2  - \norm{\wkpd - \xs}^2\right) 
\\ &\qquad \qquad  =   \frac{\alpha - 1}{k+\alpha -1}\spr{ \God(\zkd), \wkd - \xs} 
- \frac{\omega}{2} \norm{\God(\zkd)}^2 \,.
\end{split}
\end{equation*}
Replacing this in inequality \eqref{ineq_helper_1} above, we get
\begin{equation*}
\begin{split}
\Thetad(\xkpd) \leq & \frac{k}{k+\alpha-1} \Thetad(\xkd) + \frac{\alpha-1}{k+\alpha-1} \Thetad(\xs) + \Dd
\\
& + \frac{(\alpha - 1)^2}{2\omega (k+\alpha -1)^2} 
\left( \norm{\wkd - \xs}^2  - \norm{\wkpd - \xs}^2\right)  \,.
\end{split}
\end{equation*}
Equivalently, we can write this as
\begin{equation*}
\begin{split}
\Thetad(\xkpd) - \Thetad(\xs) \leq & \frac{k}{k+\alpha-1} \left(\Thetad(\xkd)  -  \Thetad(\xs)  \right) + \Dd
\\
& + \frac{(\alpha - 1)^2}{2\omega (k+\alpha -1)^2} 
\left( \norm{\wkd - \xs}^2  - \norm{\wkpd - \xs}^2\right)  \,.
\end{split}
\end{equation*}
Multiplying by $\frac{2 \omega}{\alpha-1}(k+\alpha-1)^2$, we obtain
\begin{equation*}
\begin{split}
& \frac{2 \omega}{\alpha-1}(k+\alpha-1)^2 \left(\Thetad(\xkpd) - \Thetad(\xs) \right) \leq  \frac{2 \omega}{\alpha-1}k(k+\alpha-1) \left(\Thetad(\xkd)  -  \Thetad(\xs)  \right) 
\\
& \qquad + 
\frac{2 \omega}{\alpha-1}(k+\alpha-1)^2 \Dd + (\alpha - 1) \left( \norm{\wkd - \xs}^2  - \norm{\wkpd - \xs}^2\right)  \,,
\end{split}
\end{equation*}
and therefore, since there holds
\begin{equation*}
k(k+\alpha-1) = (k+\alpha-1)^2 - k(\alpha - 3) - (\alpha-2)^2 \leq (k+\alpha-1)^2 - k(\alpha-3) \,,
\end{equation*}
we get that
\begin{equation*}
\begin{split}
& \frac{2 \omega}{\alpha-1}(k+\alpha-1)^2 \left(\Thetad(\xkpd) - \Thetad(\xs) \right) \leq  -\frac{2 \omega}{\alpha-1}k(\alpha-3) \left(\Thetad(\xkd)  -  \Thetad(\xs)  \right) 
\\
& \qquad + \frac{2 \omega}{\alpha-1}(k+\alpha-1)^2 \left(\Thetad(\xkd) -\Thetad(\xs) \right) +  
\frac{2 \omega}{\alpha-1}(k+\alpha-1)^2 \Dd 
\\
& \qquad + (\alpha - 1) \left( \norm{\wkd - \xs}^2  - \norm{\wkpd - \xs}^2\right)  \,.
\end{split}
\end{equation*}
Together with the definition \eqref{def_E} of $\Ed$, this implies
\begin{equation*}
\Ed(k+1) + \frac{2 \omega}{\alpha-1}k(\alpha-3) \left(\Thetad(\xkd)  -  \Thetad(\xs)  \right) 
\leq \Ed(k) + \frac{2 \omega}{\alpha-1}(k+\alpha-1)^2 \Dd \,,
\end{equation*}
or equivalently, after rearranging, we get
\begin{equation*}
\Ed(k+1) + \frac{2 \omega}{\alpha-1} \left( k(\alpha-3) \left( \Thetad(\xkd) -  \Thetad(\xs) \right) - (k+\alpha-1)^2 \Dd \right)
\leq \Ed(k)  \,,
\end{equation*}
which concludes the proof.
\end{proof}

Inequality \eqref{ineq_main_Ed} is the key ingredient for showing that \eqref{Nesterov}, combined with a suitable stopping rule, gives rise to a convergent regularization method. In order to derive a suitable stopping rule, note first that in the case of exact data, i.e., $\delta = 0$, inequality \eqref{ineq_main_Ed} reduces to
\begin{equation}\label{ineq_main_E}
\Ez(k+1) + \frac{2 \omega}{\alpha-1}k(\alpha-3) \left(\Thetaz(\xkz)  -  \Thetaz(\xs)  \right) 
\leq \Ez(k)  \,.
\end{equation}
Since by Assumption~\ref{ass_main} the functional $\Phiz$ is convex, the arguments used in \cite{Attouch_Peypouquet_2016} are applicable, and we can deduce the following:
\begin{theorem}\label{thm_exact}
Let Assumption~\ref{ass_main} hold, let the sequence of iterates $\xkz$ and $\zkz$ be given by \eqref{Nesterov} with exact data $y = \yd$, i.e., $\delta=0$ and let $\S$ be defined by \eqref{def_S}. Then the following statements hold:
\begin{itemize}
\item The sequence $(\Ez(k))$ is non-increasing and $\lim\limits_{k\to \infty} \Ez(k)$ exists.
\item For each $k \geq 0$, there holds 
\begin{equation*}
\norm{F(\xkz)-y}^2  \leq \frac{(\alpha-1)\Ez(0)}{ \omega(k+\alpha-2)^2} \,, \qquad \norm{\wkz-\xs}^2 \leq \frac{\Ez(0)}{\alpha-1} \,.
\end{equation*}
\item There holds
\begin{equation*}
\sum\limits_{k=1}^{\infty} k \norm{F(\xkz)-y}^2 \leq \frac{(\alpha-1)\Ez(1)}{\omega (\alpha-3)}\,,
\end{equation*}
as well as
\begin{equation*}
\sum\limits_{k=1}^{\infty} k \norm{\xkpz - \xkz}^2 \leq \frac{(\alpha-1)\Ez(1)}{\omega (\alpha-3)} \,.
\end{equation*}
\item There holds
\begin{equation*}
\liminf_{k\to \infty}\left(k^2 \ln(k) \norm{F(\xkz)-y}^2  \right)  = 0 \,,
\end{equation*}
as well as 
\begin{equation*}
\liminf_{k\to \infty}\left(k \ln(k) \norm{\xkpz-\xkz}^2  \right)  = 0 \,.
\end{equation*}
\item There exists an $\xt$ in $\S$, such that the sequence $(\xkz)$ converges weakly to $\xt$, i.e.,
\begin{equation}\label{conv_exact}
\lim_{\delta \to 0} \spr{\xkz,h} = \spr{\xt,h}  \,, \qquad \forall \, h \in \X \,.
\end{equation}
\end{itemize}
\end{theorem}
\begin{proof}
The statements follow from Facts 1-4, Remark 2 and Theorem 3 in \cite{Attouch_Peypouquet_2016}.
\end{proof}

Thanks to Theorem~\ref{thm_exact}, we now know that Nesterov's accelerated gradient method \eqref{Nesterov} converges weakly to a solution $\xt$ from the solution set $\S$ in case of exact data $y = \yd$, i.e., $\delta = 0$. 

Hence, it remains to consider the behaviour of \eqref{Nesterov} in the case of inexact data $\yd$. As mentioned above, the key for doing so is inequality \eqref{ineq_main_Ed}. We want to use it to show that, similarly to the exact data case, the sequence $(\Ed(k))$ is non-increasing up to some $k \in \N$. To do this, note first that $\Ed(k)$ is positive as long as
\begin{equation*}
\Thetad(\xkd) \geq \Thetad(\xs) \,,
\end{equation*}
which is true, as long as
\begin{equation}\label{stop_idea_1}
\norm{F(\xkd) -\yd} \geq \delta \,.
\end{equation}
On the other hand, the term 
\begin{equation}\label{term_E_middle}
\frac{2 \omega}{\alpha-1} \left( k(\alpha-3) \left( \Thetad(\xkd) -  \Thetad(\xs) \right) - (k+\alpha-1)^2 \Dd \right)
\end{equation}
in \eqref{ineq_main_Ed} is positive, as long as
\begin{equation*}
\Thetad(\xkd) -  \Thetad(\xs) \geq \frac{(k+\alpha-1)^2}{k(\alpha-3)} \Dd \,, 
\end{equation*}
which is satisfied, as long as 
\begin{equation}\label{stop_idea_2}
\norm{F(\xkd) - \yd}^2 \geq \frac{2(k+\alpha-1)^2}{k(\alpha-3)} \Dd + \delta^2  \,,
\end{equation}
which obviously implies \eqref{stop_idea_1}. These considerations suggest, given a small $\tau > 1$, to choose the stopping index $\ks = \ks(\delta,\yd)$ as the smallest integer such that
\begin{equation}\label{stoprule}
\norm{F(\xd_{\ks}) - \yd}^2 \leq \frac{2(k+\alpha-1)^2}{k(\alpha-3)} \Dd +  \tau^2 \delta^2 < \norm{F(\xkd) - \yd}^2 \,, \qquad \ks > k \,.
\end{equation}
Concerning the well-definedness of $\ks$, we are able to prove the following
\begin{lemma}\label{lem_ks_well}
Let Assumption~\ref{ass_main} hold, let the sequence of iterates $\xkd$ and $\zkd$ be given by \eqref{Nesterov} and let $c_1$ and $c_2$ be defined by \eqref{def_c1_c2}. Then the stopping index $\ks$ defined by \eqref{stoprule} with $\tau > 1$ is well-defined and there holds
\begin{equation}\label{stoprule_speed}
\ks  = \LandauO(\delta^{-1}) \,,
\end{equation}
\end{lemma}
\begin{proof}
By the definition \eqref{def_Dd} of $\Dd$ and due to
\begin{equation*}
\norm{F(\xkd)-\yd}^2 \leq \left(\norm{F(\xkd)-y} + \norm{y-\yd}\right)^2 \leq (2\ob \rho + \delta )^2 \,,
\end{equation*}
it follows from \eqref{stoprule} that for all $k < \ks$ there holds
\begin{equation*}
\frac{2(k+\alpha-1)^2}{k(\alpha-3)} (c_1 \delta + c_2 \delta^2)  +  \tau^2\delta^2 \leq (2\ob \rho + \delta )^2 \,,
\end{equation*}
which can be rewritten as
\begin{equation}\label{ineq_helper_2}
\begin{split}
\frac{(k+\alpha-1)^2}{k(\alpha-3)} (c_1 \delta + c_2 \delta^2) 
\leq 2\ob^2 \rho^2 + 2\ob \rho \delta + (1-\tau^2) \delta^2  
\leq 2\ob^2 \rho^2 + 2\ob \rho \delta\,,
\end{split}
\end{equation}
where we have used that $\tau > 1$. Since the left hand side  in the above inequality goes to $\infty$ for $k\to \infty$, while the right hand side stays bounded, it follows that $\ks$ is finite and hence well-defined for $\delta \neq 0$. Furthermore, since
\begin{equation*}
\frac{(k+\alpha-1)^2}{k(\alpha-3)} \geq  \frac{k}{2(\alpha-3)} \,,
\end{equation*}
which can see by multiplying the above inequality by $k(\alpha-3)$, and since \eqref{ineq_helper_2} also holds for $k=\ks-1$, we get
\begin{equation*}
\frac{\ks-1 }{2(\alpha-3)}  (c_1 \delta + c_2 \delta^2)   \leq 2\ob^2 \rho^2 + 2\ob \rho \delta \,.
\end{equation*}
Reordering the terms, we arrive at
\begin{equation*}
\ks    \leq 
2(\alpha-3) \left(\frac{2\ob^2 \rho^2 + 2\ob \rho \, \delta }{c_1 \delta + c_2 \delta^2} \right) + 1\,.
\end{equation*}
from which the assertion now immediately follows.
\end{proof}

The rate $\ks = \LandauO(\delta^{-1})$ given in \eqref{stoprule_speed} for the iteration method \eqref{Nesterov} should be compared with the corresponding result \cite[Corollary~2.3]{Kaltenbacher_Neubauer_Scherzer_2008} for Landweber iteration \eqref{Landweber}, where one only obtains $\ks = \LandauO(\delta^{-2})$. In order to obtain the rate $\ks = \LandauO(\delta^{-1})$ for Landweber iteration, apart from others, a source condition of the form
\begin{equation}\label{source_cond}
\xD - x_0 \in \Range(F'(\xD)^*) 
\end{equation}
has to hold, which is not required for Nesterov's accelerated gradient method \eqref{Nesterov}.

Before we turn to our main result, we first prove a couple of important consequences of \eqref{ineq_main_Ed} and the stopping rule \eqref{stoprule}.
\begin{proposition}\label{prop_conv_inex}
Let Assumption~\ref{ass_main} be satisfied, let $\xkd$ and $\zkd$ be defined by \eqref{Nesterov} and let $\Ed$ be defined by \eqref{def_E}. Assuming that the stopping index $\ks$ is determined by \eqref{stoprule} with some $\tau > 1$, then, for all $0 \leq k \leq \ks$, the sequence $(\Ed(k))$ is non-increasing and in particular, $\Ed(k) \leq \Ed(0)$. Furthermore, for all $0\leq k \leq \ks$ there holds
\begin{equation}\label{ineq_descent_inex}
\Thetad(\xkd) - \Thetad(\xs) \leq \frac{(\alpha-1) \Ed(0)}{2\omega(k+\alpha-2)^2}\,,
\end{equation} 
as well as 
\begin{equation}\label{ineq_wkd}
\norm{\wkd-\xs}^2 \leq \frac{\Ed(0)}{(\alpha-1)} \,,
\end{equation} 
and 
\begin{equation}\label{ineq_summability}
\sum\limits_{k=1}^{\ks-1} \left( k \left( \Thetad(\xkd) -  \Thetad(\xs) \right) - \frac{(k+\alpha-1)^2}{(\alpha-3)} \Dd \right)
\leq \frac{(\alpha-1)\Ed(1)}{2\omega(\alpha-3)}
 \,.
\end{equation} 
\end{proposition}
\begin{proof}
Due to the definition of the stopping rule \eqref{stoprule} and the arguments preceding it, the term \eqref{term_E_middle} is positive for all $k\leq\ks-1$. Hence, due to \eqref{ineq_main_Ed}, $\Ed(k)$ is non-increasing for all $k \leq \ks$ and in particular, $\Ed(k) \leq \Ed(0)$. From this observation, \eqref{ineq_descent_inex} and \eqref{ineq_wkd} immediately follow from the definition \eqref{def_E} of $\Ed(k)$.

Furthermore, rearranging \eqref{ineq_main_Ed} we have
\begin{equation*}
\frac{2 \omega(\alpha-3)}{\alpha-1} \left( k \left( \Thetad(\xkd) -  \Thetad(\xs) \right) - \frac{(k+\alpha-1)^2}{(\alpha-3)} \Dd \right)
\leq \Ed(k) - \Ed(k+1)  \,.
\end{equation*}
Now, summing over this inequality and using telescoping and the fact that $\Ed(\ks) \geq 0$ we immediately arrive at \eqref{ineq_summability}, which concludes the proof.
\end{proof}

From the above proposition, we are able to deduce two interesting corollaries. 
\begin{corollary}
Under the assumptions of Proposition~\ref{prop_conv_inex} there holds
\begin{equation}\label{ineq_desc_inex_2}
\norm{F(\xkd)-\yd}^2 \leq \frac{2(\alpha-1) \Ed(0)}{\omega(k+\alpha-2)^2} + \delta^2\,, \qquad 0 \leq k \leq \ks \,.
\end{equation}
\end{corollary}
\begin{proof}
Using the fact that both $\xkd, \xs \in \Btr(x_0)$, it follows from the definition of $\Thetad$ that $\Thetad(\xkd) = \Phid(\xkd)$ and $\Thetad(\xs) = \Phid(\xs)$. Hence, inequality \eqref{ineq_descent_inex} yields
\begin{equation*}
\norm{F(\xkd)-\yd}^2 \leq \frac{2(\alpha-1) \Ed(0)}{\omega(k+\alpha-2)^2} + \norm{y-\yd}^2\,, \qquad 0 \leq k \leq \ks \,,
\end{equation*}
from which, using $\norm{y-\yd}\leq \delta$, the statement immediately follows.
\end{proof}
\begin{corollary}\label{cor_ks_bound}
Under the assumptions of Proposition~\ref{prop_conv_inex} there holds
\begin{equation*}
\ks (\ks-1) 
\leq \left(\frac{2(\alpha-1)\Ed(1)}{\omega(\alpha-3)(\tau^2-1)}\right) \frac{1}{\delta^2}
 \,.
\end{equation*} 
\end{corollary}
\begin{proof}
Using the fact that both $\xkd, \xs \in \Btr(x_0)$, it follows from the definition of $\Thetad$ that $\Thetad(\xkd) = \Phid(\xkd)$ and $\Thetad(\xs) = \Phid(\xs)$
Hence, it follows with $\norm{y-\yd}\leq \delta$ that
\begin{equation*}
\begin{split}
& k \left( \Thetad(\xkd) -  \Thetad(\xs) \right) - \frac{(k+\alpha-1)^2}{(\alpha-3)} \Dd 
\\
&  \qquad \geq \frac{k}{2} \left( \norm{F(\xkd)-\yd}^2 - \delta^2 \right) - \frac{(k+\alpha-1)^2}{(\alpha-3)} \Dd \,. 
\end{split}
\end{equation*} 
Together with the definition of the stopping rule \eqref{stoprule}, this implies that for all $k \leq \ks-1$
\begin{equation*}
\begin{split}
k \left( \Thetad(\xkd) -  \Thetad(\xs) \right) - \frac{(k+\alpha-1)^2}{(\alpha-3)} \Dd  > k\frac{(\tau^2-1)\delta^2 }{2}
\end{split}
\end{equation*}
Using this in \eqref{ineq_summability} yields
\begin{equation*}
\frac{(\tau^2-1)\delta^2 }{2} \sum\limits_{k=1}^{\ks-1} k
\leq \frac{(\alpha-1)\Ed(1)}{2\omega(\alpha-3)}
 \,,
\end{equation*} 
from which the statement now immediately follows.
\end{proof}
Again, this shows that $\ks = \LandauO(\delta^{-1})$, i.e., $\ks \leq c \delta^{-1}$, however this time the constant $c$ does not depend on $c_1$ and $c_2$, an observation which we use when analysing \eqref{Nesterov} under slightly different assumptions then Assumption~\ref{ass_main} below.

We are now able to prove one of our main results:
\begin{theorem}\label{thm_main}
Let Assumption~\ref{ass_main} hold and let the iterates $\xkd$ and $\zkd$ be defined by \eqref{Nesterov}. Furthermore, let $\ks = \ks(\delta,\yd)$ be determined by \eqref{stoprule} with some $\tau > 1$ and let the solution set $\S$ be given by \eqref{def_S}. Then there exists an $\xt \in \S$ and a subsequence $\xksdt$ of $\xksd$ which converges weakly to $\xt$ as $\delta \to 0$, i.e.,
\begin{equation*}
\lim_{\delta \to 0} \spr{\xksdt,h} = \spr{\xt,h} \,, \qquad \forall \,  h \in \X \,.
\end{equation*}
If $\S$ is a singleton, then $\xksd$ converges weakly to the then unique solution $\xt \in \S$.
\end{theorem}
\begin{proof}
This proof follows some ideas of \cite{Hanke_Neubauer_Scherzer_1995}. Let $y_n := \ydn$ be a sequence of noisy data satisfying $\norm{y-y_n}\leq \dn$. Furthermore, let $k_n := \ks(\dn,y_n)$ be the stopping index determined by \eqref{stoprule} applied to the pair $(\dn,y_n)$. There are two cases. First, assume that $k$ is a finite accumulation point of $k_n$. Without loss of generality, we can assume that $k_n=k$ for all $n \in \N$. Thus, from \eqref{stoprule}, it follows that
\begin{equation*}
\norm{F(x_k^\dn)-y_n} < \frac{2(k+\alpha-1)^2}{k(\alpha-3)} \Delta(\dn) +  \tau^2 \dn^2 \,,
\end{equation*} 
which, together with the triangle inequality, implies
\begin{equation*}
\norm{F(x_k^\dn)- y} \leq \norm{F(x_k^\dn)-y_n} + \norm{y_n-y} \leq \frac{2(k+\alpha-1)^2}{k(\alpha-3)} \Delta(\dn) +  \tau^2 \dn^2  + \dn
\end{equation*} 
Since for fixed $k$ the iterates $\xkd$ depend continuously on the data $\yd$, by taking the limit $n \to \infty$ in the above inequality we can derive
\begin{equation*}
x_k^\dn \to \xkz \,, 
\qquad 
F (x_k^\dn) \to F(\xkz) = y\,, \quad \text{as } n \to \infty \,. 
\end{equation*}

For the second case, assume that $k_n \to \infty$ as $n \to \infty$. Since $x_{k_n}^\dn \in \Btr(x_0)$, it is bounded and hence, has a weakly convergent subsequence $\xkndnt$, corresponding to a subsequence $\dnt$ of $\dn$ and $\knt := \ks(\dnt,y^\dnt)$. Denoting the weak limit of $\xkndnt$ by $\xt$, it remains to show that $\xt \in \S$. For this, observe that it follows from \eqref{ineq_desc_inex_2} that
\begin{equation*}
\norm{F(\xkndnt)-\yd}^2 \leq \frac{2(\alpha-1) \mathcal{E}^\dnt(0)}{\omega(\knt+\alpha-2)^2} + \dnt^2 \longrightarrow 0 \,, \quad \text{as } n \to \infty \,.
\end{equation*}
where we have used that $\knt \to \infty$ and $\dnt \to 0$  as $n \to \infty$, which follows from the assumption that so do the sequences $k_n$ and $\dn$, and the fact that $\Ed(0)$ stays bounded for $\delta \to 0$. Hence, since we know that $\yd \to y$ as $\delta \to 0$, we can deduce that
\begin{equation*}
F\kl{\xkndnt} \to y \,, \quad \text{as } n \to \infty \,,
\end{equation*}
and therefore, using the weak sequential closedness of $F$ on $\Btr(x_0)$, we deduce that $F(\xt) = y$, i.e., $\xt \in \S$, which was what we wanted to show.

It remains to show that if $\S$ is a singleton then $\xksd$ converges weakly to $\xt$. Since this was already proven above in the case that $k_n$ has a finite accumulation point, it remains to consider the second case, i.e., $k_n \to \infty$. For this, consider an arbitrary subsequence of $\xksd$. Since this sequence is bounded, it has a weakly convergent subsequence which, by the same arguments as above, converges to a solution $\xt \in \S$. However, since we have assumed that $\S$ is a singleton, it follows that $\xksd$ converges weakly to $\xt$, which concludes the proof.
\end{proof}

\begin{remark}
In Theorem~\ref{thm_main}, we have shown weak subsequential convergence to an element $\xt$ in the solution set $\S$. However, this element might be different from the $x_0$-minimum norm solution $\xD$ defined by \eqref{def_xD}, unless of course in case that $\S$ is a singleton.
\end{remark}

\section{Convergence Analysis II}

Some simplifications of the above presented convergence analysis are possible if we assume that instead of only $\Phiz$, all the functionals $\Phid$ are convex. Hence, for the remainder of this section, we work with the following
\begin{assumption}\label{ass_main_2}
Let $\rho$ be a positive number such that $\Bmr(x_0) \subset \D(F)$.
\begin{enumerate}
\item The operator $F : \D(F) \subset \X \to \Y$ is continuously Fr\'echet differentiable between the real Hilbert spaces $\X$ and $\Y$ with inner products $\spr{.,.}$ and norms $\norm{.}$. Furthermore, let $F$ be weakly sequentially closed on $\Btr(x_0)$.
\item The equation $F(x) = y$ has a solution $\xs \in \Br(x_0)$.
\item The data $\yd$ satisfies $\norm{y-\yd} \leq \delta$.
\item The functionals $\Phid$ are convex and have Lipschitz continuous gradients $\nabla \Phid$ with uniform Lipschitz constant $L$ on $\Bmr(x_0)$, i.e.,
\begin{equation}\label{cond_convex_2}
\Phid(\lambda x_1 + (1-\lambda) x_2) \leq \lambda \Phid(x_1) + (1-\lambda) \Phid(x_2)  
\,, 
\qquad \forall x_1, x_2 \in \Bmr(x_0) \,,
\end{equation} 
\begin{equation*}
\norm{\nabla \Phid(x_1) - \nabla \Phid(x_2)} \leq L \norm{x_1 - x_2}
\,,
\qquad \forall x_1, x_2 \in \Bmr(x_0) \,.
\end{equation*}
\item For $\alpha$ in \eqref{Nesterov} there holds $\alpha > 3$ and the scaling parameter $\omega$ satisfies $0 < \omega < \frac{1}{L}$.
\end{enumerate}
\end{assumption}

Note that Assumption~\ref{ass_main_2} is only a special case of Assumption~\ref{ass_main}. Hence, the above convergence analysis presented above is applicable and we get weak convergence of the iterates of \eqref{Nesterov}. However, the stopping rule \eqref{stoprule} depends on the constants $c_1$ and $c_2$ defined by \eqref{def_c1_c2}, which are not always available in practise. Fortunately, using the Assumption~\ref{ass_main_2}, we can get rid of $c_1$ and $c_2$. The key idea is to observe that the following lemma holds:
\begin{lemma}\label{lem_main_noise_2}
Under Assumption~\ref{ass_main_2}, for all $x,z \in \Bmr(x_0)$ there holds
\begin{equation*}
\Thetad(z-\omega\God(z)) \leq \Thetad(x) + \spr{\God(z),z-x} - \frac{\omega}{2} \norm{\God(z)}^2 \,.
\end{equation*}
\end{lemma}
\begin{proof}
This follows from the convexity of $\Thetad$ in the same way as in Lemma~\ref{lem_main_zero}.
\end{proof}
From the above lemma, it follows that the results of Corollary~\ref{cor_main_noise} and Proposition~\ref{prop_descent} hold with $\Dd = 0$. Therefore, the stopping rule \eqref{stoprule} simplifies to
\begin{equation}\label{stoprule_2}
\norm{F(\xd_{\ks}) - \yd} \leq  \tau \delta < \norm{F(\xkd) - \yd} \,, \qquad \ks \geq k \,,
\end{equation}
for some $\tau > 1$, which is nothing else than the discrepancy principle \eqref{discrepancy_principle}. Note that in contrast to \eqref{stoprule}, only the noise level $\delta$ needs to be known in order to determine the stopping index $\ks$. With the same arguments as above, we are now able to prove our second main result:
\begin{theorem}\label{thm_main_2}
Let Assumption~\ref{ass_main_2} hold and let the iterates $\xkd$ and $\zkd$ be defined by \eqref{Nesterov}. Furthermore, let $\ks = \ks(\delta,\yd)$ be determined by \eqref{stoprule_2} with some $\tau > 1$ and let the solution set $\S$ be given by \eqref{def_S}. Then for the stopping index $\ks$ there holds $\ks = \LandauO(\delta^{-1})$. Furthermore, there exists an $\xt \in \S$ and a subsequence $\xksdt$ of $\xksd$ which converges weakly to $\xt$ as $\delta \to 0$, i.e.,
\begin{equation*}
\lim_{\delta \to 0} \spr{\xksdt,h} = \spr{\xt,h} \,, \qquad \forall \,  h \in \X \,.
\end{equation*}
If $\S$ is a singleton, then $\xksd$ converges weakly to the then unique solution $\xt \in \S$.
\end{theorem}
\begin{proof}
The proof of this theorem is analogous to the proof of Theorem~\ref{thm_main}. The only main difference is the well definedness of $\ks$, which now cannot be derived from Lemma~\ref{lem_ks_well} but follows from \eqref{ineq_summability} by Corollary~\ref{cor_ks_bound}, which also yields $\ks = \LandauO(\delta^{-1})$. 
\end{proof}

\begin{remark}
Note that since Theorem~\ref{thm_main_2} only gives an asymptotic result, i.e., for $\delta \to 0$, the requirement in Assumption~\ref{ass_main_2} that the functionals $\Phid$ have to be convex for all $\delta > 0$ can be relaxed to $0 \leq \delta \leq \bd$, as long as we only consider data $\yd$ satisfying the noise constraint $\norm{y-\yd} \leq \delta \leq \bd$.
\end{remark}
\begin{remark}
Note that if the functionals $\Phid$ are globally convex and uniformly Lipschitz continuous, which is for example the case if $F$ is a bounded linear operator, then one can choose $\rho$ arbitrarily large in the definition of $\Psi$. Now, as we have seen above, the proximal mapping $\prox{.}$ is nothing else than the projection onto $\Btr(x_0)$. This implies that for practical purposes, $\prox{.}$ may be dropped in \eqref{Nesterov}, which means that one effectively uses \eqref{Nesterov_1} instead of \eqref{Nesterov}.
\end{remark}

\section{Strong Convexity and Nonlinearity Conditions}\label{sect_strong_conv}

In this section, we consider the question of strong convergence of the iterates of \eqref{Nesterov} and comment on the connection between the assumption of local convexity and the (weak) tangential cone condition.

Concerning the strong convergence of the iterates of \eqref{Nesterov} and \eqref{Nesterov_1}, note that it could be achieved if the functional $\Phiz$ were locally strongly convex, i.e., if
\begin{equation}\label{conv_strong_local}
\begin{split}
\spr{F'(x_1)^*(F(x_1)-y) - F'(x_2)^*(F(x_2)-y),x_1-x_2} \geq \alpha \norm{x_1 - x_2}^2 \,, 
\\
\forall \, x_1,x_2 \in \Btr(x_0) \,,
\end{split}
\end{equation} 
since then, for the choice of $x_1 = \xkz$ and $x_2 = \xs$, one gets
\begin{equation*}
\alpha \norm{\xkz - \xs} \leq \spr{F'(\xkz)^*(F(\xkz)-y),\xkz - \xs} \leq 2\ob\rho \norm{F(\xkz)-y} \,,
\end{equation*}
from which, since we have $\norm{F(\xkz)-y} \to 0$ as $\delta \to 0$, it follows that $\xkd$ converges strongly to $\xs$ as $\delta \to 0$. Hence, retracing the proof of Theorem~\ref{thm_main}, one would get
\begin{equation*}
\lim\limits_{\delta \to 0} \xksd = \xs \,.
\end{equation*}
Unfortunately, already for linear ill-posed operators $F = A$, strong convexity of the form \eqref{conv_strong_local} cannot be satisfied, since then one would get
\begin{equation*}
\norm{Ax_1 - Ax_2} \geq \alpha \norm{x_1 - x_2} \,, \qquad \forall \, x_1 , x_2 \in \Btr(x_0) \,,
\end{equation*}
which already implies the well-posedness of $Ax = y$ in $\Btr(x_0)$. However, defining 
\begin{equation}\label{def_Mt}
\Mt(A) := \left\{ x \in \Btr \, \vert \, \exists w \in \Y \,, \norm{w} \leq \tau \,, \, x - \xD = A^*w   \right\} \,,
\end{equation}
it was shown in \cite[Lemma~3.3]{Hofmann_Scherzer_1998} that there holds
\begin{equation*}
\norm{x - \xD}^2 \leq \tau \norm{A x - A \xD} \,, \qquad \forall \, x \in  \Mt(A) \,,
\end{equation*}
Hence, if one could show that $\xkz \in \Mt$ for some $\tau > 0$ and all $k \in \N$, then it would follow that
\begin{equation*}
\norm{\xkz - \xD}^2 \leq \tau \norm{A \xkz - y} \,, \qquad \forall \, x \in  \Mt(A) \,,
\end{equation*}
from which strong convergence of $\xkz$, and consequently also of $\xksd$ to $\xD$ would follow. In essence, this was done in \cite{Neubauer_2017} with tools from spectral theory in the classical framework for analysing linear ill-posed problem \cite{Engl_Hanke_Neubauer_1996} under the source condition $\xD \in \Range(A^*)$.

\begin{remark}
Note that it is sometimes possible, given weak convergence of a sequence $x_k \in \X$ to some element $\xt \in \X$, to infer strong convergence of $x_k$ to $\xt$ in a weaker topology. For example, if $x_k \in \HoO$ converges weakly to $\xt$ in the $\HoO$ norm, then it follows that $x_k$ converges strongly to $\xt$ with respect to the $\LtO$ norm. Many generalizations of this example are possible. Note further that in finite dimensions, weak and strong convergence coincide.
\end{remark}

In the remaining part of this section, we want to comment on the connection of the local convexity assumption \eqref{cond_convex} to other nonlinearity conditions like \eqref{cond_nonlin} and \eqref{cond_nonlin_weak} commonly used in the analysis of nonlinear-inverse problems.

First of all, note that due to the results of Kindermann \cite{Kindermann_2017}, we know that both convexity and the (weak) tangential cone condition imply weak convergence of Landweber iteration \eqref{Landweber}. However, it is not entirely clear in which way those conditions are connected.

One connection of the two conditions was given in \cite{Scherzer_1995}, where it was shown that the nonlinearity condition implies a certain directional convexity condition. Another connection was provided in \cite{Kindermann_2017}, where it was shown that the tangential cone condition implies a quasi-convexity condition. However, it is not clear whether or not the tangential cone condition implies convexity or not. What we can say is that convexity does not imply the (weak) tangential cone condition, which is shown in the following

\begin{example}
Consider the operator $F: H^1[0,1] \to L^2[0,1]$ defined by
	\begin{equation}
		F(x)(s):= \int_{0}^{s} x(t)^2 \, dt \,.
	\end{equation}
This nonlinear Hammerstein operator was extensively treated as an example problem for nonlinear inverse problems (see for example \cite{Hanke_Neubauer_Scherzer_1995, Neubauer_2016}). It is well known that for this operator the tangential cone condition is satisfied around $\xD$ as long as $\xD \geq c > 0$. However, the (weak) tangential cone condition is not satisfied in case that $\xD \equiv 0$. Moreover, it can easily be seen (for example from \eqref{ineq_conv}) that $\Phiz(x)$ is globally convex, which shows that convexity does not imply the tangential cone condition.
\end{example}

\section{Example Problems}

In this section, we consider two examples to which we apply the theory developed above. Most importantly, we prove the local convexity assumption for both $\Phiz$ and $\Phid$, with $\delta$ small enough. Furthermore, based on these example problems, we present some numerical results, demonstrating the usefulness of method \eqref{Nesterov}, and supporting the findings of \cite{Jin_2016, Hubmer_Ramlau_2017, Neubauer_2017, Hubmer_Neubauer_Ramlau_Voss_2018,Hubmer_Sherina_Neubauer_Scherzer_2018}, which are also shortly discussed.

For this, note that if $F$ is twice continuously Fr\'echet differentiable, then convexity of $\Phid$ is equivalent to positive semi-definiteness of its second Fr\'echet derivative \cite{Rockafellar_Wets_Wets_2009}. More precisely, we have that \eqref{cond_convex_2} is equivalent to 
	\begin{equation}\label{ineq_conv}
		\norm{F'(x)h}^2 + \spr{F(x)-\yd,F''(x)(h,h)} \geq 0 \,, \qquad  \forall \, x \in \Bmr(x_0) \,, \, \forall \, h \in \D(F) \,,
	\end{equation}
which is our main tool for the upcoming analysis.

\subsection{Example 1 - Nonlinear Diagonal Operator}

For our first (academic) example, we look at the following class of nonlinear diagonal operators
	\begin{equation*}
		F: \ell^2 \to \ell^2 \,,
		\qquad 
		x:=(x_n)_{n\in \N} \mapsto \sum\limits_{n=1}^{\infty}  f_n(x_n) \, e_n 
	\end{equation*}
where $\kl{e_n}_{n\in\N}$ is the canonical orthonormal basis of $\ell^2$. These operators are reminiscent of the singular value decomposition of compact linear operators. Here we consider the special choice
	\begin{equation}\label{ex1_def_fn}
		f_n(z) := 
		\frac{1}{n} \cdot \begin{cases}
		z^2 \,, & n \leq M \,,
		\\
		z \,, & n > M \,,
		\end{cases}
	\end{equation}
for some fixed $M > 0$. For this choice, $F$ takes the form
	\begin{equation*}
		F(x) = \sum\limits_{n=1}^{M} \frac{1}{n}\, x_n^2 e_n  + \sum\limits_{n= M+1}^{\infty} \frac{1}{n}\, x_n e_n \,.
	\end{equation*}
It is easy to see that $F$ is a well-defined, twice continuously Fr\'echet differentiable operator with
	\begin{equation*}
	\begin{split}
		F'(x)h &= 2\,  \sum\limits_{n=1}^{M} \frac{1}{n}\, x_n h_n e_n  + 2 \sum\limits_{n= M+1}^{\infty} \frac{1}{n}\, h_n e_n \,,
		\\
		F''(x)(h,w) &= 2 \, \sum\limits_{n=1}^{M} \frac{1}{n}\, h_n w_n e_n \,.
	\end{split}
	\end{equation*} 
Furthermore, note that solving $F(x) = y$ is equivalent to 
	\begin{equation*}
		x_n = 
		n 
		\begin{cases}
		\sqrt{y_n} \,, & n \leq M \,,
		\\
		y_n \,, & n > M \,,
		\end{cases}
	\end{equation*}
from which it is easy to see that we are dealing with an ill-posed problem.

We now turn to the convexity of $\Phid(x)$ around a solution $\xD$.
\begin{proposition}
Let $\xD$ be a solution of $F(x) = y$ such that $\abs{\xD_n} > 0$ holds for all $n \in \{1\,,\dots\,,M\}$. Furthermore, let $\rho > 0$ and $\bd \geq 0$ be small enough such that
	\begin{equation}\label{ex1_ineq_main}
		(\xD_n)^2 \geq 28 \vert \xD_n \vert \rho  + \bd \kl{2 \norm{y}_{\lt} + \bd} \,,
		\qquad \forall \, n \in \kl{1\,,\dots\,,M} \,,
	\end{equation}
and let $x_0 \in \Br(\xD)$. Then for all $0 \leq \delta \leq \bd$, the functional $\Phid(x)$ is convex in $\Bmr(x_0)$.
\end{proposition}
\begin{proof}
Due to \eqref{ineq_conv} it is sufficient to show that
	\begin{equation*}
	\begin{split}
		0 & \leq \norm{F'(x)h}^2 + \spr{F(x)-\yd,F''(x)(h,h)} =
		\\
		& 		= \norm{F'(x)h}^2 +  \spr{F(x)-y,F''(x)(h,h)} +  \spr{\yd - y,F''(x)(h,h)} 
	\end{split}
	\end{equation*}
Using the definition of $F$, the fact that $e_n$ is an orthonormal basis of $\ell^2$ and that $F(\xD) = y$, this inequality can be rewritten into 
	\begin{equation*}
	\begin{split}
		2\kl{\sum\limits_{n=1}^{M}\frac{1}{n^2} x_n^2 h_n^2  
		+ \sum\limits_{n=M+1}^{\infty}\frac{1}{n^2} h_n^2  }    
		+ 2 \sum\limits_{n=1}^M (x_n^2 - (\xD_n)^2)h_n^2
		+ 2 \sum\limits_{n=1}^M (y_n^2 - (\yd_n)^2)h_n^2 \geq 0 \,,
	\end{split}
	\end{equation*}
which after simplification, becomes
	\begin{equation*}
		2 \sum\limits_{n=1}^{M}h_n^2\kl{2 x_n^2 - (\xD_n)^2 + y_n^2 - (\yd_n)^2} 
		+ 
		2 \sum\limits_{n=M+1}^{\infty}\frac{1}{n^2} h_n^2 \geq 0 
	\end{equation*}
Since the right of the above two sums is always positive, in order for the above inequality to be satisfied it suffices to show that
	\begin{equation}\label{ex1_cond_conv}
		2 x_n^2 - (\xD_n)^2 + y_n^2 - (\yd_n)^2 \geq 0 \,,
		\qquad \forall \, n \in \{1\,,\dots\,,M\} \,.
	\end{equation}
Now, since by the triangle inequality we have
	\begin{equation}\label{ex1_ineq_y}
	\begin{split}
		\abs{y_n^2 -(\yd_n)^2} 
		&= \abs{y_n - \yd_n} \abs{y_n + \yd_n}
		\leq \norm{y - \yd}_{\ell^2} \norm{y + \yd}_{\ell^2}
		\\
		& \leq 
		\delta \left( 2 \norm{y}_\lt + \norm{y - \yd}_\lt \right)
		\leq \delta \kl{2 \norm{y}_\lt + \delta} \,,
	\end{split}
	\end{equation}
it follows that in order to prove \eqref{ex1_cond_conv} it suffices to show
	\begin{equation*}
		2 x_n^2 - (\xD_n)^2 - \delta \kl{2 \norm{y}_\lt + \delta}  \geq 0 \,,
		\qquad \forall \, n \in \{1\,,\dots\,,M\} \,.
	\end{equation*}
Now, writing $x = \xD + \eps$, this can be rewritten into 	
	\begin{equation*}
		(\xD_n)^2 + 4 \xD_n \eps_n + 2 \eps_n^2 - \delta \kl{2 \norm{y}_\lt + \delta}  \geq 0 \,,
		\qquad \forall \, n \in \{1\,,\dots\,,M\} \,.
	\end{equation*}
Since $\eps_n^2 \geq 0$, the above inequality is satisfied given that
	\begin{equation*}
		(\xD_n)^2 - 4 \abs{\xD_n} \abs{\eps_n}  - \delta \kl{2 \norm{y}_\lt + \delta}  \geq 0 \,,
		\qquad \forall \, n \in \{1\,,\dots\,,M\}\,.
	\end{equation*}
However, since $\abs{\eps_k} \leq \norm{\eps}_{\ell^2} = \norm{x- \xD}_{\lt} \leq \norm{x - x_0}_{\lt}  + \norm{x_0 - \xD}_{\lt}  \leq 7 \rho$, this follows immediately from \eqref{ex1_ineq_main}, which concludes the proof.
\end{proof}	

\begin{remark}
Due to $\abs{\xD_k} \leq \norm{\xD_k}_\lt$, condition \eqref{ex1_ineq_main} is satisfied given that
	\begin{equation*}
		\min\limits_{n =1\,,\dots\,,M} \left\{(\xD_n)^2\right\} \geq 28 \norm{\xD}_\lt \rho + \bd\kl{\norm{y}_\lt + \bd} \,,
	\end{equation*}
which can always be satisfied given that $\abs{\xD_n} > 0$ for all $n \in \{1\,,\dots\,,M\}$. 
\end{remark}

After proving local convexity of the residual functional around the solution, we now proceed to demonstrate the usefulness of \eqref{Nesterov} based on the following numerical

\begin{example}\label{ex1} For this example we choose $f_n$ as in \eqref{ex1_def_fn} with $M=100$. For the exact solution $\xD$ we take the sequence $\xD_n = 100/n$ which leads to the exact data 
	\begin{equation*}
		y_n = F(\xD)_n = 
		\begin{cases}
		10^4/n^3 \,, & n \leq 100 \,,
		\\
		10^2/n^2 \,, & n > 100 \,.
		\end{cases}
	\end{equation*}
Hence, condition \eqref{ex1_cond_conv} reads as follows 
	\begin{equation*}
		10^4/n^2 \geq 28 (10^2/ n) \rho + \bd(2 \norm{y}_\lt + \bd) \,, \qquad \forall \, n \in \{1\,, \dots\,, 100\} \,.
	\end{equation*}
Therefore, the functional $\Phiz$ is convex in $\Bmr(x_0)$ given that $\rho \leq 1/28 \approx 0.036$, which is for example the case for the choice
	\begin{equation}\label{ex1_def_x0}
		x_0 = \xD + \kl{(-1)^n \frac{ \rho \sqrt{6}}{\pi n} }_{n \in \N}\,.
	\end{equation} 
Furthermore, for any noise level $\bd$ small enough, one has that for all $\delta \leq \bd$ the functional $\Phid$ is convex in $\Bmr(x_0)$ as long as
	\begin{equation*}
		\rho \leq \frac{10^4/n - n \, \bd(2\norm{y}_\lt + \bd)}{2800} \,, \qquad \forall \, n \in \{1\,, \dots\,, 100\} \,,
	\end{equation*}
which for example is satisfied if 
	\begin{equation*}
		\rho \leq \frac{1 - \bd(2\norm{y}_\lt + \bd)}{28} \,.
	\end{equation*}
	
For numerically treating the problem, instead of considering full sequences $x = \kl{x_n}_{n \in \N}$, we only consider $\xv = \kl{x_n}_{n=1,\dots,N}$ where we choose $N=200$ in this example. This means that we are considering the following discretized version of $F$:
	\begin{equation*}
		F_n(\xv) = \sum\limits_{n=1}^{100} \frac{1}{n}\, x_n^2 e_n  + \sum\limits_{n= 101}^{200} \frac{1}{n}\, x_n e_n \,.
	\end{equation*}
We now compare the behaviour of method \eqref{Nesterov} with its non-accelerated Landweber counterpart \eqref{Landweber} when applied to the problem with $\xD$ and $x_0$ as defined above. For both methods, we choose the same scaling parameter $\omega = 3.2682*10^{-5}$ estimated from the norm of $F(\xD)$ and we stop the iteration with the discrepancy principle \eqref{discrepancy_principle} with $\tau = 1$. Furthermore, random noise with a relative noise level of $0.001\%$ was added to the data to arrive at the noisy data $\yd$ and, following the argument presented after \eqref{thm_main_2} and since the iterates $\xkd$ remain bounded even without it, we drop the proximal operator $\prox{.}$ in \eqref{Nesterov}. The results of the experiments, computed in MATLAB, are displayed in Table~\ref{table_ex1}. The speedup both in time and in the number of iterations achieved by Nesterov's acceleration scheme is obvious. Not only does \eqref{Nesterov} satisfy the discrepancy principle much earlier than \eqref{Landweber}, but also the relative error is even a bit smaller for method \eqref{Nesterov}.

\begin{table}[ht!]
\centering
\begin{tabular}{| c | c | c | c |}
\hline 
\rule{0pt}{15pt} Method & \multicolumn{1}{c|}{$k_*$}  & \multicolumn{1}{c|}{Time}  & $\norm{\xD - \xkd }/\norm{\xD}$ \\
\hline
Landweber & 82 & 0.057 s & 0.0109 \% \\
\hline
Nesterov & 23 & 0.019 s & 0.0108 \% \\
\hline
\end{tabular} 
\caption{Comparison of Landweber iteration \eqref{Landweber} and its Nesterov accelerated version \eqref{Nesterov} when applied to the diagonal operator problem considered in Example~\ref{ex1}.}
\label{table_ex1}
\end{table}

\end{example}

\subsection{Example 2 - Auto-Convolution Operator}

Next we look at an example involving an auto-convolution operator. Due to its importance in laser optics, the auto-convolution problem has been extensively studied in the literature \cite{Anzengruber_Buerger_Hofmann_Steinmeyer_2016,Birkholz_Steinmeyer_Koke_Gerth_Buerger_Hofmann_2015, Gerth_2014}, its ill-posedness has been shown in \cite{Buerger_Hofmann_2015,Fleischer_Hofmann_1996, Gorenflo_Hofmann_1994} and its special structure was successfully exploited in \cite{Ramlau_2003}. For our purposes, we consider the following version of the auto-convolution operator
	\begin{equation}\label{ex2_def_F}
		F: \LtO \to \LtO \,, 
		\qquad 
		F(x)(s):= 
		(x \ast x) (s) 
		:= \int\limits_{0}^{1} x(s-t)x(t)\,dt \,,
	\end{equation}
where we interpret functions in $\LtO$ as $1$-periodic functions on $\R$. For the following, denote by $(\ek)_{k\in \Z}$ the canonical real Fourier basis of $\LtO$, i.e., 
	\begin{equation*}
		\ek(t) := 
		\begin{cases}
		1 \,, & k = 0 \,,
		\\
		\sqrt{2} \sin(2\pi k t) \,, & k \geq 1 \,,
		\\
		\sqrt{2} \cos(2\pi k t) \,, & k \leq -1 \,,
		\end{cases}
		\qquad t \in (0,1) \,,
	\end{equation*}
and by $x_k := \spr{x,\ek}$ the Fourier coefficients of $x$. It follows that
	\begin{equation}\label{ex2_eq_conv}
		x \ast w = \sum\limits_{k\in \Z} x_k w_k \ek \,.
	\end{equation}
It was shown in \cite{Buerger_Flemming_Hofmann_2016} that if only finitely many Fourier components $x_k$ are non-zero, then a variational source condition is satisfied leading to convergence rates for Tikhonov regularization. We now use this assumption of a sparse Fourier representation to prove convexity of $\Phid$ for the auto-convolution operator in the following 

\begin{proposition}
Let $\xD$ be a solution of $F(x) = y$ such that there exists an index set $\LN \subset \Z$ with $\abs{\LN} = N$ such that for the Fourier coefficients $\xD_k$ of $\xD$ there holds
	\begin{equation*}
		\xD_k = 0 \,, \qquad \forall \, k \in \Z \setminus \LN \,.
	\end{equation*} 
Furthermore, let $\rho > 0$ and $\bd \geq 0$ be small enough such that
	\begin{equation}\label{ex2_ineq_main}
		(\xD_k)^2 \geq 28 \vert \xD_k\vert \, \rho  + \bd \kl{2 \norm{y}_{\Lt} + \bd} \,, \qquad \forall \, k \in \LN
	\end{equation}
and let $x_0 \in \Br(\xD)$. Then for all $0 \leq \delta \leq \bd$, the functional $\Phid(x)$ is convex in $\Bmr(x_0)$.
\end{proposition}
\begin{proof}
As in the previous example, we want to show that \eqref{ineq_conv} is satisfied, which, due to \eqref{ex2_eq_conv} and the fact that the $\ek$ form an orthonormal basis is equivalent to
	\begin{equation*}
		\sum_{k\in\Z} x_k^2 h_k^2 
		+ 
		\sum_{k\in\Z} (x_k^2 - (\xD_k)^2) h_k^2  
		+ 
		\sum_{k\in\Z} ((\yd_k)^2 - y_k^2)h_k^2 
		\geq 0 \,,
	\end{equation*}
which, after simplification, becomes
	\begin{equation*}
		\sum_{k\in\Z}  h_k^2 \kl{ 2 x_k^2 
		 - (\xD_k)^2  
		+ (\yd_k)^2 - y_k^2 }
		\geq 0 \,,
	\end{equation*}
and hence, it is sufficient to show that
	\begin{equation}\label{ex2_cond_conv}
		 2 x_k^2 - (\xD_k)^2 + (\yd_k)^2 - y_k^2 \geq 0 \,,
		 \qquad \forall \, k \in \Z \,.  
	\end{equation}
Note that this is essentially the same condition as \eqref{ex1_cond_conv} in the previous example, apart from that here we have to show the inequality for all $k \in \Z$. However, if $k \notin \LN$, then $\xD_k = y_k = 0$ and hence, \eqref{ex2_cond_conv} is trivially satisfied. Hence, it remains to prove \eqref{ex2_cond_conv} only for $k\in\LN$. For this, we write $x_k = \xD_k + \eps_k$, which allows us to rewrite \eqref{ex1_cond_conv} into 
	\begin{equation*}
		(\xD_k)^2 + 4 \xD_k \eps_k + 2\eps_k^2 +  (\yd_k)^2 - y_k^2 \geq 0  \qquad \forall \, k \in \LN \,. 
	\end{equation*}
Now since we get as in \eqref{ex1_ineq_y} that $\abs{y_k^2 - (\yd_k)^2} \leq \delta \kl{2\norm{y}_\Lt + \delta}$, it follows that for the above inequality to be satisfied, it suffices to have 
	\begin{equation*}
		(\xD_k)^2 - 4 \abs{\xD_k}\abs{\eps_k} - \delta \kl{2\norm{y}_\Lt + \delta} \geq 0 
		\,, \qquad \forall \, k \in \LN \,.
	\end{equation*}
However, since $\abs{\eps_k} \leq \norm{\eps}_\Lt = \norm{x-\xD} ß\leq \norm{x- x_0} + \norm{x_0 - \xD}\leq 7 \rho$, this immediately follows from \eqref{ex2_ineq_main}, which completes the proof.
\end{proof}

\begin{remark}
Similarly to the previous example, condition \eqref{ex1_ineq_main} is satisfied given that
	\begin{equation*}
		\min\limits_{k \in \LN } \left\{(\xD_k)^2\right\} \geq 28 \norm{\xD}_\Lt \rho + \bd\kl{\norm{y}_\lt + \bd} \,,
	\end{equation*}
which can always be satisfied given that $\abs{\xD_n} > 0$ for all $n \in \{1\,,\dots\,,M\}$. 
\end{remark}

\begin{remark}
Note that one could also consider $F$ as an operator from $\HoO \to \LtO$, in which case the local convexity of $\Phid$ is still satisfied. Since, as noted in Section~\ref{sect_strong_conv}, weak convergence in $\HoO$ implies strong convergence in $\LtO$, the convergence analysis carried out in the previous section then implies strong subsequential $\LtO$ convergence of the iterates $\xkd$ of \eqref{Nesterov} to an element $\xt \in \S$ from the solution set.
\end{remark}

\begin{example}\label{ex2}
For this example, we consider the auto-convolution problem with exact solution $\xD(s) := 10 + \sqrt{2} \sin(2\pi s)$. It follows that
	\begin{equation*}
		\xD_k = \spr{\xD,\ek} 
		=
		\begin{cases}
		10 \,, & k = 0\,,
		\\
		1 \,, & k = 1\,,
		\\
		0 \,, & \text{else} \,.
		\end{cases}
	\end{equation*}
and therefore, the convexity condition \eqref{ex2_ineq_main} simplifies to the following two inequalities 
	\begin{equation*}
	\begin{split}
		100 \geq 280 \rho + \bd\kl{2\norm{y}_\Lt + \bd} \,,
		\qquad 
		1 \geq 28 \rho + \bd\kl{2\norm{y}_\Lt + \bd} \,.
	\end{split}
	\end{equation*}
Hence, for the noise-free case (i.e., $\bd=0$) the functional $\Phiz$ is convex in $\Bmr(x_0)$ given that $\rho \leq 1/28 \approx 0.036$ and that $x_0 \in \Br(\xD)$, which is for example the case for the choice $x_0 = 10 + \frac{27}{28}\sqrt{2}\sin(2\pi s)$. 

For discretizing the problem, we choose a uniform discretization of the interval $[0,1]$ into $N = 32$ equally spaced subintervals and introduce the standard finite element hat functions $\{\psi_i \}_{i=0}^N$ on this subdivision, which we use to discretize both $\X$ and $\Y$. Following the idea used in \cite{Neubauer_2000}, we discretize $F$ by the finite dimensional operator
	\begin{equation}\label{ex2_helper_int}
		F_N(x)(s) := \sum\limits_{i=0}^{N} f_i(x) \psi_i(s) \,,
		\qquad \text{where} \qquad
		f_i(x) := \int\limits_{0}^{1}x\kl{\tfrac{i}{N} - t}x(t) \, dt \,.
	\end{equation}
For computing the coefficients $f_i(x)$, we employ a $4$-point Gaussian quadrature rule on each of the subintervals to approximate the integral in \eqref{ex2_helper_int}.  

Now we again compare method \eqref{Nesterov} with  \eqref{Landweber}. This time, the estimated scaling parameter has the value $\omega = 0.005$ and random noise with a relative noise level of $0.01\%$ was added to the data. Again the discrepancy principle \eqref{discrepancy_principle} with $\tau = 1$ was used and the proximal operator $\prox{.}$ in \eqref{Nesterov} was dropped. The results of the experiments, computed in MATLAB, are displayed in the left part of Table~\ref{table_ex2}. Again the results clearly illustrate the advantages of Nesterov's acceleration strategy, which substantially decreases the required number of iterations and computational time, while leading to a relative error of essentially the same size as Landweber iteration.

The initial guess $x_0$ used for the experiment above is quite close to the exact solution $\xD$. Although this is necessary for being able to guarantee convergence by our developed theory, it is not very practical. Hence, we want to see what happens if the solution and the initial guess are so far apart that they are no longer within the guaranteed area of convexity. For this, we consider the choice of $\xD(s) = 10 + \sqrt{2}\sin\kl{8 \pi s}$ and $x_0(s) = 10 + \sqrt{2}\sin\kl{2 \pi s}$. The result can be seen in the right part of Table~\ref{table_ex2}. Landweber iteration was stopped after $10000$ iterations without having reached the discrepancy principle since no more progress was visible numerically. Consequently, it is clearly outperformed by \eqref{Nesterov}, which manages to converge already after $797$ iterations, and with a much better relative error. The resulting reconstructions, depicted in Figure~\ref{fig_ex2}, once again underline the usefulness of \eqref{Nesterov}. 

As an interesting remark, note that it seems that for the second example Landweber iteration gets stuck in a local minimum, while \eqref{Nesterov}, after staying at this minimum for a while, manages to escape it, which is likely due to the combination step in \eqref{Nesterov}.

\begin{table}[ht!]
\centering
\resizebox{0.48\columnwidth}{!}{
\begin{tabular}{| c | c | c | c |}
\hline 
\rule{0pt}{15pt} Method & \multicolumn{1}{c|}{$k_*$}  & \multicolumn{1}{c|}{Time}  & $\norm{\xD - \xkd }/\norm{\xD}$ \\
\hline
Landweber & 526 & 57 s & 0.0244 \% \\
\hline
Nesterov & 50 & 6 s & 0.0271 \% \\
\hline
\end{tabular} }
\quad
\resizebox{0.48\columnwidth}{!}{
\begin{tabular}{| c | c | c | c |}
\hline 
\rule{0pt}{15pt} Method & \multicolumn{1}{c|}{$k_*$}  & \multicolumn{1}{c|}{Time}  & $\norm{\xD - \xkd }/\norm{\xD}$ \\
\hline
Landweber & 10000 & 1067 s & 9.57 \% \\
\hline
Nesterov & 797 & 87 s &  0.65\% \\
\hline
\end{tabular} }
\caption{Comparison of Landweber iteration \eqref{Landweber} and its Nesterov accelerated version \eqref{Nesterov} when applied to the auto-convolution problem considered in Example~\ref{ex2} for the choice $\xD(s) = 10 + \sqrt{2}\sin\kl{2 \pi s}$ and $x_0(s) =  10 + \tfrac{27}{28}\sqrt{2}\sin\kl{2 \pi s}$ (left table) and $\xD(s) = 10 + \sqrt{2}\sin\kl{8 \pi s}$ and $x_0(s) = 10 + \sqrt{2}\sin\kl{2 \pi s}$ (right table).}
\label{table_ex2}
\end{table}

	\begin{figure}[ht!]
		\includegraphics[width=0.7\textwidth]{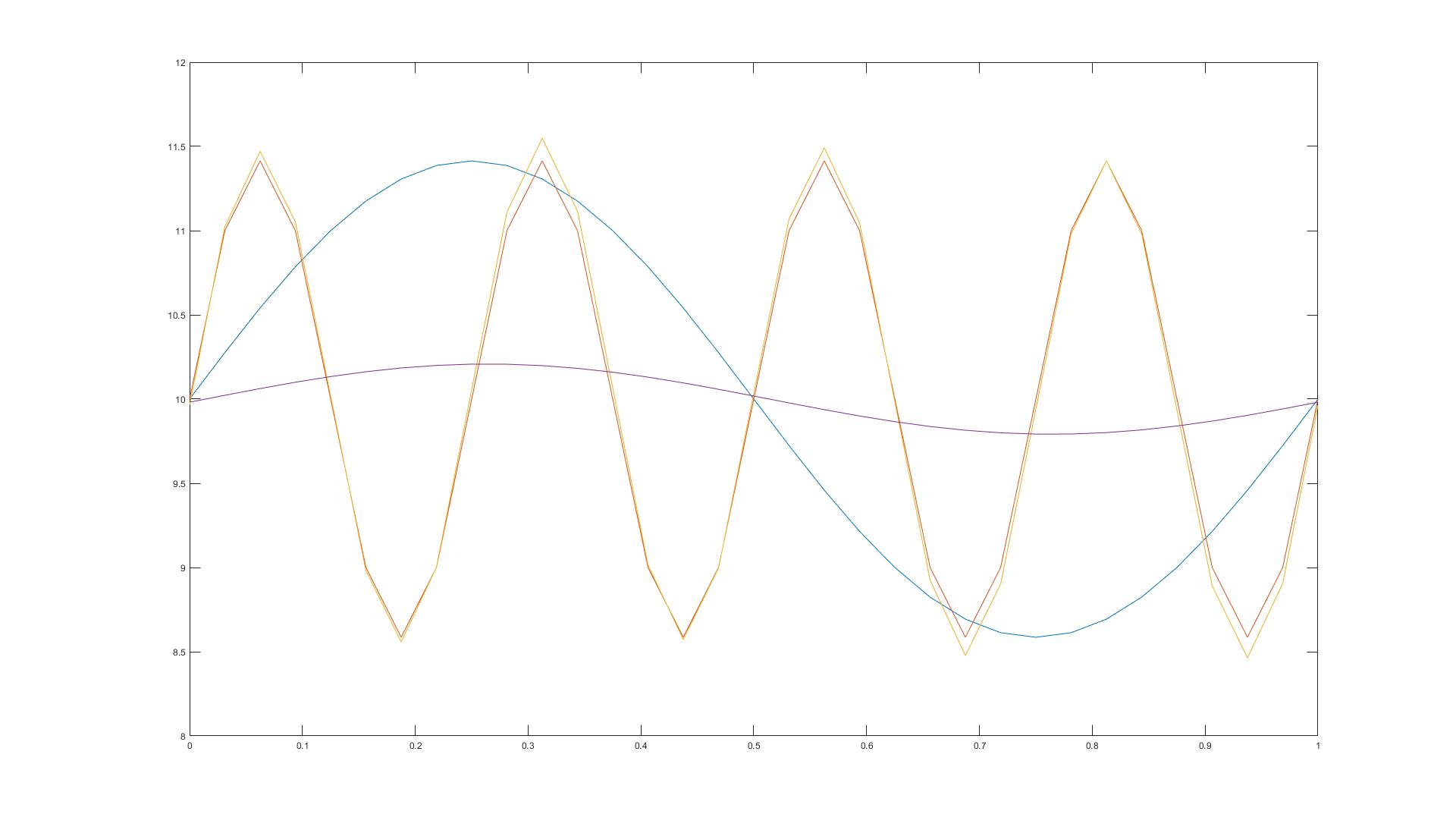}
		\centering
		\caption{Auto-convolution example: Initial guess $x_0$ (blue), exact solution $\xD$ (red), Landweber \eqref{Landweber} reconstruction (purple), Nesterov \eqref{Nesterov} reconstruction (yellow).}
		\label{fig_ex2}
	\end{figure}

\end{example}

\subsection{Further Examples}

Besides the two rather academic examples presented above, we would like to cite a number of other examples where methods like \eqref{Nesterov_1} and \eqref{Nesterov} were successfully used, even though the key assumption of local convexity is not always known to hold for them.

First of all, in \cite{Hubmer_Neubauer_Ramlau_Voss_2018} the parameter estimation problem of Magnetic Resonance Advection Imaging (MRAI) was solved using a method very similar to \eqref{Nesterov}. In MRAI, one aims at estimating the spatially varying pulse wave velocity (PWV) in blood vessels in the brain from Magnetic Resonance Imaging (MRI) data. The PWV is directly connected to the health of the blood vessels and hence, it is used as a prognostic marker for various diseases in medical examinations. The data sets in MRAI are very large, making the direct application of second order methods like \eqref{LM} or \eqref{irGN} difficult. However, since methods like \eqref{Nesterov} can deal with those large datasets, they were used in \cite{Hubmer_Neubauer_Ramlau_Voss_2018} for reconstructions of the PWV.

Secondly, in \cite{Hubmer_Ramlau_2017}, numerical examples for various TPG methods \eqref{Two_Point}, including the iteration \eqref{Nesterov_1}, were presented. Among those is an example based on the imaging technique of Single Photon Emission Computed Tomography (SPECT). Various numerical tests show that among all tested TPG methods, the method \eqref{Nesterov_1} clearly outperforms the rest, even though the local convexity assumption is not known to hold in this case. This is also demonstrated on an example based on a nonlinear Hammerstein operator.

Thirdly, method \eqref{Nesterov_1} was used in \cite{Hubmer_Sherina_Neubauer_Scherzer_2018} to solve a problem in Quantitative Elastography, namely the reconstruction of the spatially varying Lam\'e parameters from full internal static displacement field measurements. Method \eqref{Nesterov_1} was used to obtain all reconstruction results presented in that paper, since ordinary first-order methods like Landweber iteration \eqref{Landweber} were too slow to satisfy the demands required in practise.

Finally, in the numerical examples presented in \cite{Jin_2016}, method \eqref{Nesterov_1} was used to accelerate the employed gradient/Kaczmarz methods. Furthermore, a convergence analysis of \eqref{Nesterov_1} for linear ill-posed problems including numerical examples is given in \cite{Neubauer_2017}.

\section{Support and Acknowledgements}
The authors were partly funded by the Austrian Science Fund (FWF): W1214-N15, project DK8 and F6805-N36, project 5. Furthermore, they would like to thank Dr.\ Stefan Kindermann and Prof.\ Andreas Neubauer for providing valuable suggestions and insights during discussions of the subject.

\bibliographystyle{plain}
{\footnotesize
\bibliography{mybib}

\begin{thebibliography}{10}

\bibitem{Anzengruber_Buerger_Hofmann_Steinmeyer_2016}
S.~W. Anzengruber, S.~B\"urger, B.~Hofmann, and G.~Steinmeyer.
\newblock Variational regularization of complex deautoconvolution and phase
  retrieval in ultrashort laser pulse characterization.
\newblock {\em Inverse Problems}, 32(3):035002, 2016.

\bibitem{Attouch_Peypouquet_2016}
H.~{Attouch} and J.~{Peypouquet}.
\newblock The {R}ate of {C}onvergence of {N}esterov's {A}ccelerated
  {F}orward--{B}ackward {M}ethod is {A}ctually {F}aster {T}han $o(1/k^2)$.
\newblock {\em SIAM Journal on Optimization}, 26(3):1824--1834, 2016.

\bibitem{Bauschke_Combettes_2017}
H.~H. Bauschke and P.~L. Combettes.
\newblock {\em Convex analysis and monotone operator theory in {H}ilbert
  spaces}, volume 2011.
\newblock Springer, 2017.

\bibitem{Beck_Teboulle_2009}
A.~Beck and M.~Teboulle.
\newblock {A Fast Iterative Shrinkage-Thresholding Algorithm for Linear Inverse
  Problems.}
\newblock {\em SIAM J. Imaging Sci.}, 2(1):183--202, 2009.

\bibitem{Birkholz_Steinmeyer_Koke_Gerth_Buerger_Hofmann_2015}
S.~Birkholz, G.~Steinmeyer, S.~Koke, D.~Gerth, S.~B\"{u}rger, and B.~Hofmann.
\newblock Phase retrieval via regularization in self-diffraction-based spectral
  interferometry.
\newblock {\em J. Opt. Soc. Am. B}, 32(5):983--992, May 2015.

\bibitem{Blaschke_Neubauer_Scherzer_1997}
B.~Blaschke, A.~Neubauer, and O.~Scherzer.
\newblock {On convergence rates for the Iteratively regularized Gauss-Newton
  method}.
\newblock {\em IMA Journal of Numerical Analysis}, 17(3):421, 1997.

\bibitem{Buerger_Flemming_Hofmann_2016}
S.~B\"urger, J.~Flemming, and B.~Hofmann.
\newblock On complex-valued deautoconvolution of compactly supported functions
  with sparse {F}ourier representation.
\newblock {\em Inverse Problems}, 32(10):104006, 2016.

\bibitem{Buerger_Hofmann_2015}
S.~B\"urger and B.~Hofmann.
\newblock About a deficit in low-order convergence rates on the example of
  autoconvolution.
\newblock {\em Applicable Analysis}, 94(3):477--493, 2015.

\bibitem{Engl_Hanke_Neubauer_1996}
H.~W. {Engl}, M.~{Hanke}, and A.~{Neubauer}.
\newblock {\em {Regularization of inverse problems.}}
\newblock Dordrecht: Kluwer Academic Publishers, 1996.

\bibitem{Fleischer_Hofmann_1996}
G.~Fleischer and B.~Hofmann.
\newblock On inversion rates for the autoconvolution equation.
\newblock {\em Inverse Problems}, 12(4):419, 1996.

\bibitem{Gerth_2014}
D.~Gerth, B.~Hofmann, S.~Birkholz, S.~Koke, and G.~Steinmeyer.
\newblock Regularization of an autoconvolution problem in ultrashort laser
  pulse characterization.
\newblock {\em Inverse Problems in Science and Engineering}, 22(2):245--266,
  2014.

\bibitem{Gorenflo_Hofmann_1994}
R.~Gorenflo and B.~Hofmann.
\newblock On autoconvolution and regularization.
\newblock {\em Inverse Problems}, 10(2):353, 1994.

\bibitem{Hanke_1991}
M.~Hanke.
\newblock Accelerated landweber iterations for the solution of ill-posed
  equations.
\newblock {\em {Numerische Mathematik}}, 60(1):341--373, 1991.

\bibitem{Hanke_1997}
M.~Hanke.
\newblock {A regularizing Levenberg - Marquardt scheme, with applications to
  inverse groundwater filtration problems}.
\newblock {\em Inverse Problems}, 13(1):79, 1997.

\bibitem{Hanke_Neubauer_Scherzer_1995}
M.~Hanke, A.~Neubauer, and O.~Scherzer.
\newblock {A convergence analysis of the Landweber iteration for nonlinear
  ill-posed problems}.
\newblock {\em Numerische Mathematik}, 72(1):21--37, 1995.

\bibitem{Hofmann_Scherzer_1998}
B.~Hofmann and O.~Scherzer.
\newblock Local ill-posedness and source conditions of operator equations in
  hilbert spaces.
\newblock {\em Inverse Problems}, 14(5):1189, 1998.

\bibitem{Hubmer_Neubauer_Ramlau_Voss_2018}
S.~Hubmer, A.~Neubauer, R.~Ramlau, and H.~U. Voss.
\newblock On the parameter estimation problem of magnetic resonance advection
  imaging.
\newblock {\em Inverse Problems and Imaging}, 12(1):175--204, 2018.

\bibitem{Hubmer_Ramlau_2017}
S.~Hubmer and R.~Ramlau.
\newblock Convergence analysis of a two-point gradient method for nonlinear
  ill-posed problems.
\newblock {\em Inverse Problems}, 33(9):095004, 2017.

\bibitem{Hubmer_Sherina_Neubauer_Scherzer_2018}
S.~Hubmer, E.~Sherina, A.~Neubauer, and O.~Scherzer.
\newblock Lam\'e {P}arameter {E}stimation from {S}tatic {D}isplacement {F}ield
  {M}easurements in the {F}ramework of {N}onlinear {I}nverse {P}roblems.
\newblock {\em SIAM Journal on Imaging Sciences}, 2018.
\newblock accepted.

\bibitem{Jin_2010}
Q.~Jin.
\newblock {On a regularized Levenberg--Marquardt method for solving nonlinear
  inverse problems}.
\newblock {\em Numerische Mathematik}, 115(2):229--259, 2010.

\bibitem{Jin_2016}
Q.~Jin.
\newblock {Landweber-Kaczmarz method in Banach spaces with inexact inner
  solvers}.
\newblock {\em Inverse Problems}, 32(10):104005, 2016.

\bibitem{Jin_Tautenhahn_2009}
Q.~Jin and U.~Tautenhahn.
\newblock {On the discrepancy principle for some Newton type methods for
  solving nonlinear inverse problems}.
\newblock {\em Numerische Mathematik}, 111(4):509--558, 2009.

\bibitem{Kaltenbacher_Neubauer_Scherzer_2008}
B.~{Kaltenbacher}, A.~{Neubauer}, and O.~{Scherzer}.
\newblock {\em {Iterative regularization methods for nonlinear ill-posed
  problems.}}
\newblock Berlin: de Gruyter, 2008.

\bibitem{Kindermann_2017}
S.~{Kindermann}.
\newblock Convergence of the gradient method for ill-posed problems.
\newblock {\em Inverse Problems and Imaging}, 11(4):703--720, 2017.

\bibitem{Nesterov_1983}
Y.~Nesterov.
\newblock {A method of solving a convex programming problem with convergence
  rate $O(1/k^2)$.}
\newblock {\em Soviet Mathematics Doklady}, 27(2):372--376, 1983.

\bibitem{Neubauer_2000}
A.~{Neubauer}.
\newblock On {L}andweber iteration for nonlinear ill-posed problems in
  {H}ilbert scales.
\newblock {\em {Numer. Math.}}, 85(2):309--328, 2000.

\bibitem{Neubauer_2016}
A.~Neubauer.
\newblock {Some generalizations for Landweber iteration for nonlinear ill-posed
  problems in Hilbert scales}.
\newblock {\em Journal of Inverse and Ill-posed Problems}, 24(4):393--406,
  2016.

\bibitem{Neubauer_2017}
A.~Neubauer.
\newblock On {N}esterov acceleration for {L}andweber iteration of linear
  ill{-}posed problems.
\newblock {\em J. Inv. Ill-Posed Problems}, 25(3):381--390, 2017.

\bibitem{Ramlau_2003}
{R. Ramlau}.
\newblock {TIGRA - an iterative algorithm for regularizing nonlinear ill-posed
  problems}.
\newblock {\em Inverse Problems}, 19(2):433, 2003.

\bibitem{Ramlau_1999}
R.~Ramlau.
\newblock A modified {L}andweber method for inverse problems.
\newblock {\em Numerical Functional Analysis and Optimization}, 20(1-2):79--98,
  1999.

\bibitem{Rockafellar_Wets_Wets_2009}
R.~T. Rockafellar, M.~Wets, and T.~J.~B. Wets.
\newblock {\em Variational Analysis}.
\newblock Grundlehren der mathematischen Wissenschaften. Springer Berlin
  Heidelberg, 2009.

\bibitem{Scherzer_1995}
O.~Scherzer.
\newblock {Convergence Criteria of Iterative Methods Based on Landweber
  Iteration for Solving Nonlinear Problems}.
\newblock {\em Journal of Mathematical Analysis and Applications},
  194(3):911--933, 1995.

\end{thebibliography}
}

\end{document}